%% file: main_TN_revision.tex
\documentclass[11pt, a4paper, oneside, reqno]{amsart}

\input{style/style_PM_2}
\input{style/commands_PM}
\graphicspath{{images/}}

\title[]{Multiple Faults Estimation in Dynamical Systems: Tractable Design and Performance Bounds}

\author[C.J. van der Ploeg]{Chris van der Ploeg}
\author[M. Alirezaei]{Mohsen Alirezaei}
\author[N. van de Wouw]{Nathan van de Wouw}
\author[P. {Mohajerin Esfahani}]{Peyman {Mohajerin Esfahani}}

\thanks{
The authors are with the Department of Mechanical Engineering, Eindhoven University of Technology,
($\{${\tt C.J.v.d.Ploeg, M.Alirezaei, N.v.d.Wouw@tue.nl$\}$@tue.nl}), and the Delft Center for Systems and Control, Delft University of Technology ({\tt P.MohajerinEsfahani@tudelft.nl}).} %
\thanks{Peyman Mohajerin Esfahani acknowledges the
support of the ERC grant TRUST-949796.}

\begin{document}
	\maketitle
	\begin{abstract}   
	    We propose a tractable nonlinear fault estimation filter along with explicit performance bounds for a class of linear dynamical systems in the presence of both additive and non-linear multiplicative faults. We consider the case where both faults may occur simultaneously and through an identical dynamical relationship, a setting that is relevant to several application domains including automotive driving, aviation, and chemical plants. The proposed filter architecture combines tools from model-based approaches in the control literature and regression techniques from machine learning. To this end, we view the regression operator through a \mbox{system-theoretic} perspective to develop operator bounds that are then utilized to derive performance bounds for the proposed estimation filter. In the case of constant, simultaneously, and identically acting additive and multiplicative faults, it can be shown that the estimation error converges to zero with an exponential rate. The performance of the proposed estimation filter in the presence of incipient faults is validated through an application on the lateral safety systems of SAE level 4 automated vehicles. The numerical results show that the theoretical bounds of this study are indeed close to the actual estimation error.
	
	\end{abstract}

\newlength\figH
\newlength\figW
\setlength{\figH}{7cm}
\setlength{\figW}{12cm}	
\section{Introduction} \label{sec:introduction}

In the detection task, the objective is to detect the presence of a fault in real-time while being insensitive to natural disturbances~\cite{beard1971failure} and/or model uncertainty~\cite{7428855} to prevent false positives. Considering the occurrence of multiple faults at the same time, we typically refer to isolation as the task to identify which one of the faults occurs. A classical approach toward isolation is to treat the problem as a special case of detection in which all the possible faults are viewed as natural disturbances.
This methodology is found in a great variety of model settings, e.g., from single non-linear systems~\cite{DU20141066} to multi-agent, possibly large-scale,   systems~\cite{Shames20112757,KELIRIS2015408,Boem20194}. A more integral approach for detection and isolation of faults is an unknown-input type estimator, which decouples the effect of unknown state measurements and disturbances (or faults) from the~\textit{residual} through an algebraic approach~\cite{2006Nyberg, Esfahani2016} or approaches using the generalized inverse~\cite{Ansari201911, Zhan20182146}.

A follow-up step to fault detection and isolation is fault-tolerance (or fault-resilience) control in which the objective is to counteract and mitigate the faults in real-time. To this end, {\em estimation} of the exact value of the fault signal is a vital aspect. When assuming a linear or linearized system description, additive faults can be estimated using standard system-theoretic tools~\cite{GERTLER1997653}. If, however, the fault is multiplicative, the fault estimation is a more challenging task due to the nonlinear impact of the fault. A possible approach to deal with multiplicative dynamics is to borrow tools from the machine learning literature (e.g., regression~\cite{HOFLING19961361}), or by reformulating multiplicative faults as additive faults~\cite{LAN201648}. The combined {\em estimation} problem of both additive and multiplicative faults, acting on the same system, can be considered a form of simultaneous state and parameter estimation. This problem is relevant in a broad range of application domains (e.g., automotive as illustrated later in this work, aviation~\cite{LIU20193849} and chemical plants~\cite{TAO201933}) where actuators/sensors, which can inhibit simultaneously a bias (i.e., an additive fault) or loss-of-effectiveness (i.e., a multiplicative fault)~\cite{8370129}, are used in safety-critical applications. This problem has been considered in several different settings, an example of which is the simultaneous appearance of multiplicative input faults and additive output faults~\cite{Chen20202107}, i.e., the faults are assumed to appear linearly independent. Other works consider additive and multiplicative faults acting through the same dynamical relationship (i.e., linearly dependent)~\cite{8370129}. The problem is however largely unexplored when both additive and multiplicative fault types act simultaneously in the system while assuming this linear dependence between the faults. 

The central problem, defined and solved in this manuscript, is to {\em{estimate}} the fault signals (rather than only acknowledging/detecting their {\em presence}) in {\em{real-time}} when both additive and multiplicative faults are present and act {\em simultaneously} through {\em{identical}} dynamical relationships. Due to the dynamical inseparability of the additive and multiplicative faults, the estimates of such faults will, by definition, be affected by one another. It is therefore vital to determine an explicit performance bound that quantifies these effects. In this light, the following problem is the main focus of this study.
\begin{pprb}\label{prob:1}
   Consider a linear dynamical system with the available measurement signal~$z$ and the multivariate signal~$f=[\fa,\fm]$ comprising possibly both additive ($\fa$) and multiplicative ($\fm$) faults that are not dynamically separable. We aim to design an estimation filter that turns the signal~$z$ to an estimation signal~$\widehat{f}=[\widehat{f}_a,\widehat{f}_m]$ (i.e., a causal dynamic mapping $z \mapsto \widehat{f}$)  such that the additive and multiplicative faults can be estimated separately, where the combined estimation error $\| f - \widehat{f}\|_2$ is bounded by
    \begin{align} 
    \label{perf}
    \| f(k) - \widehat{f}(k)\|_2 \le \constant{}(C_{z}, C_{f},k-k_0),
    \end{align}
    where the constant~$\constant{}$ is an explicit bound depending on the dynamical model, the parameters~$C_{z}$ and $C_{f}$ representing characteristics of the measurement $z$ and fault signals $f$, and the time difference~$k - k_0$ in which $k_0$ denotes the discrete starting sample of the fault signal~$f$ and $k$ is the current time instance. The signals characteristics can, for instance, include the information of the average and variance of the respective signals. 
\end{pprb}
Let us emphasize that a performance bound in the form of~\eqref{perf} provides a real-time estimation error for each and every single element of the multivariate signal~$f$.

{
{\bf Our contributions:} The distinct feature of the problem above that makes it particularly challenging is the combination of three aspects: (i) real-time estimation of a multivariate fault signal, (ii) the presence of inseparable\footnote{{See Section~\ref{subsec:separability} for a more precise definition of this terminology.}} additive and multiplicative faults, and hence a (particular) form of nonlinear dynamics, (iii) explicit, rigorous performance bounds for fault detection. To our best of knowledge, no approach in the existing literature addresses all these aspects at the same time. Our proposed solution method leverages concepts from the system theory literature with regard to the aspects (i) and (iii) while using tools from the machine learning literature to deal with the aspect (ii). This combination yields an estimation filter with three components as depicted in Figure~\ref{fig:schemthm1}. More specifically, the technical contributions of this work are summarized as follows:

\begin{enumerate}[label=(\roman*)]
    \item We develop system-theoretic (error) bounds for the regression operator, a well-known scheme borrowed from the machine learning literature~(Proposition~\ref{prop:regression}). These bounds play a crucial role to quantify the performance of the proposed estimation filter.
    \item We propose a general estimation architecture as in Figure~\ref{fig:schemthm1} that comprises three intertwined components. When the component pre-filter is a simple identity operation, we develop an explicit, computable performance bound in terms of the average and variance of the fault signals~(Theorem~\ref{theorem:main:I}). In the special case of constant faults, the proposed performance bound provides insight regarding the convergence and boundedness of the estimation error~(Corollary~\ref{cor:constant faults:I}). 
    
    \item Building on the insight obtained from the performance of the static pre-filter, we propose an alternative design utilizing a dynamic pre-filter and develop the corresponding performance bound~(Theorem~\ref{theorem:main:II}). We further show that in the special case of constants faults the estimation error decays to zero exponentially fast~(Corollary~\ref{cor:constant faults:II}). 

\end{enumerate}}

Furthermore, we also develop two technical results concerning the output bounds of linear time-invariant systems with zero steady-state gain~(Lemma~\ref{lem:LTI}) and the variance of product signals (Lemma~\ref{lem:varbound}) that facilitate the proof of the main results highlighted above. While these results admittedly seem standard, we however did not find them in the literature in the present form as needed for the main results of this study. The proof of these lemmas is relegated to an extended version of this work~\cite{vanderploeg2020multiple} due to the space limitation.

The remaining part of this paper is organized as follows. Section~\ref{sec:problem} introduces a detailed problem description 
 and challenges of the research topic; furthermore, an outline of the proposed approach is given. Following the problem description, the main theoretical results of the work are given in Section~\ref{sec:results}. The theoretical results are backed by technical proofs, which are given in Section~\ref{sec:technicalproofs}. In Section~\ref{sec:casestudy}, the theoretical results are accompanied by numerical simulations, showing the contributions of the set of developed theorems in more practical daylight. 
\begin{nnot}
    The symbols $\N$ and $\R$ represent the set of integers and real numbers and the symbol $\R_+$ represents the set of non-negative real numbers. The ones column vector with the length $n$ is denoted by~$\ind_{n}:=\left[1,1,\hdots,1\right]\tr$. The $p$-norm of a vector~$v$ is denoted by $\|v\|_p$ where $p \in [1,\infty]$. Given a square matrix~$A$ with strictly real eigenvalues, we denote by~$\overline{\lambda}\in\mathbb{R}$ and $\underline{\lambda}\in\mathbb{R}$ the maximum and minimum eigenvalue values of the matrix, respectively. Given a matrix $A\in\R^{n\times m}$, its transpose is denoted by~$A^{\intercal}\in\R^{m\times n}$, the norm~$\lVert A\rVert_2 = \overline{\sigma}(A) = \sqrt{\overline{\lambda}(A^{\intercal}A)}$ is the largest singular value, and $A^\dagger \Let (A\tr A)^{-1}A\tr$ is the pseudo-inverse. Given two matrices with an equal dimension~$A, B \in \R^{m\times n}$, the operator~$A\circ B \in \R^{m\times n}$ denotes the element-wise (also known as Hadamard) product of two matrices. The operators $\avek{x}$ and $\vark{x}$ map $\R$-valued discrete-time signals to $\R$-valued discrete-time signals, and are defined as the first moment $\ave{x}\Let\frac{1}{n}\sum_{i=0}^{n-1}x(k-i)$ and the centered second moment  $\varn[2]{x}{n}\Let\frac{1}{n}\sum_{i=0}^{n-1}x^2(k-i)-\ave[2]{x}$ of the signal $x$ over the last $n$ time instants. Throughout this study we reserve the bold sub-scripted by~$n$ $\mathbf{x}_n$ as the concatenated version of the signal~$x$ over the last~$n$~time instants: $\vecsig{x} \Let \begin{bmatrix}x(k),x(k-1),\hdots,x(k-n+1)\end{bmatrix}^{\intercal}$. The symbol $\shiftop$ represents the shift operator, i.e., $\shiftop[x(k)]=x(k+1)$. 
 \end{nnot}
\section{Problem description and outline of the proposed Approach} \label{sec:problem}
In this section, a formal description of the generic model class along with the basic principles of existing FDI schemes is given. Using this class of models, the high-level problem can be formulated. We further elaborate on the challenges and shortcomings of the current literature. Finally, an outline of the proposed solution is provided, addressing the challenges in the preceding parts.

\subsection{Model description and problem statement}
Throughout this study we consider dynamical systems described via a discrete-time non-linear differential-algebraic equation (DAE) of the form,
\begin{align}\label{eq:DAE}
	H(\shiftop)[x]+L(\shiftop)[z]+F(\shiftop)\big[\fa + \EZo\fm\big]=0,
\end{align}
where $x, z, \fa, \fm$ represent discrete-time signals, indexed by the the counter $k$ (e.g., $x(k)$), taking values in $\R^{n_x}, \R^{n_z}, \R^{n_f}$, respectively.The mapping~$E:\R^{n_z}\ra\R^{n_E}$ is a static algebraic mapping capturing the nonlinearity of the fault dynamics, which is assumed known and, depending on the application, can be obtained through first-principle modeling (see Section~\ref{sec:casestudy}). The dependency on the signal $z$ of the mapping $E$ can be extended to other unknown signals $x$ through the use of additional state estimators. Let $n_r$ represent the number of rows in~\eqref{eq:DAE}, and the matrices $H(\shiftop), L(\shiftop), F(\shiftop)$ are polynomial functions with $n_r$ rows and $n_x,\:n_z,\:n_f$ columns in the variable $\shiftop$, which represents the shift operator. As such, these matrices may be cast as linear operators in the space of discrete-time signals. 
The signal $x$ contains all unknown signals in the DAE system, typically comprising the internal states and unknown exogenous disturbances. The signal $z$ is composed of all known signals including the control inputs $u$ and the output measurements $y$. The signal $\fa$ represents an additive fault while the signal $\fm$ is considered to be a multiplicative fault or intrusion which interacts non-linearly with the signal $E(z)$. The overall contribution of both fault signals can then be seen in the term~$\fa + \EZo\fm$, to which we may refer as the ``{\em aggregated fault signal}" hereafter. Note that, for the sake of generality in this work, the location of the faults $\fa$ and $\fm$ is not restricted to any particular location and hence could represent amongst others the notions of, e.g., sensor faults or actuator faults as widely adopted in the FDI literature.

The modeling framework~\eqref{eq:DAE} encompasses a large class of dynamical systems. A motivating example to show its level of generality is the set of nonlinear ordinary difference equations (ODE) described by
	\begin{align}\label{eq:ldif}
	\begin{cases}
		GX(k+1) = AX(k)+B_uu(k)+B_dd(k)+B_{f}\Big(\fa(k)+E_X\big(B_XX(k),u(k)\big)\fm(k)\Big),\\
		y(k) = CX(k)+D_uu(k)+D_dd(k)+D_{f}\Big(\fa(k)+E_Y\big(B_YX(k),u(k)\big)\fm(k)\Big),
	\end{cases}
	\end{align}
where $u$ is the input signal, $d$ the unknown exogenous disturbance, $X$ the internal state of the system, $Y$ the measurable output, $\fa$ the additively acting set of faults or intrusions and finally $\fm$ the set of faults acting as a multiplication on a non-linear combination of the internal states and input. The matrices~$G$, $A$, $B_u$, $B_d$, $B_{f}$, $B_X$, $B_Y$, $C$, $D_u$, $D_d$ and $D_{f}$ are constant matrices with appropriate dimensions. The following fact provides a simple-to-check condition under which the ODE model~\eqref{eq:ldif} falls into the category of our nonlinear DAE model~\eqref{eq:DAE}.

\begin{Fact}[From ODE to DAE]\label{fact:odedae}
    Consider the ODE~\eqref{eq:ldif} and suppose there exist matrices~$K_X, K_Y$ such that
	\begin{align}\label{eq:Fact:matrix}
	\begin{cases}
		B_X = K_XC, \quad  K_XD_{f}=0, \quad K_X D_d = 0, \\
		B_Y = K_YC, \quad K_YD_f=0, \quad K_YD_{d}=0.
	\end{cases}
	\end{align}
	Then, the ODE model can be viewed as a DAE model~\eqref{eq:DAE} by introducing 
	\begin{align*}
	x &= \begin{bmatrix}X\\d\end{bmatrix}, \quad z = \begin{bmatrix}y\\u\end{bmatrix}, \quad \EZo = \begin{bmatrix}E_X(K_X(Y-D_u u),u)\\E_Y(K_Y(Y-D_u u),u)\end{bmatrix}, \\
	H(\shiftop) &= \begin{bmatrix}-\shiftop G+A&B_d\\C&D_d\end{bmatrix}, \quad L(\shiftop) = \begin{bmatrix}0&B_u\\-I&D_u\end{bmatrix}, \quad  F(\shiftop) = \begin{bmatrix}B_{f}\\D_{f}\end{bmatrix}.
\end{align*}
\end{Fact}

Note that from a computational point of view checking the existence condition in~\eqref{eq:Fact:matrix} is a linear programming (LP) problem, which can be certified highly efficiently. 

Throughout this study the following assumption holds, which serves as a necessary and sufficient condition for the detectability of the aggregated fault signal~$\fa + E(z) \fm$ in~\eqref{eq:DAE}.
\begin{As}[Detectability]\label{assu:detectability}
The polynomial matrices $H(\shiftop)$ and $F(\shiftop)$ in~\eqref{eq:DAE} satisfy the rank condition
	$\text{Rank}\,\big\{\big[H(\shiftop), F(\shiftop)\big]\big\} > \text{Rank}\,\big\{H(\shiftop)\big\}.$
For simplicity of the exposition, we further assume that $F(\shiftop)$ is a polynomial column vector, i.e., $n_{\fa} = n_{\fm}=1$. 
\end{As}

Assumption~\ref{assu:detectability} paves the way to acknowledge whether the aggregated fault signal is nonzero. However, differentiating the exact contribution between additive fault~$\fa$ and the multiplicative fault introduces nontrivial challenges that we shall discuss in the next section.

\subsection{Current approach and open challenges}\label{subsec:separability}

A first step towards the problem of fault isolation is fault detection. To this end, a classical approach is to characterize the {\em behavioral} set of healthy trajectories. Namely, in absence of the fault signals $\fa, \fm$, all possible $z$-trajectories of the system can be characterized as
\begin{align}\label{eq:beh1}
	\mathcal{M}\Let\big\{z:\N\ra\R^{n_z} ~|~ \exists x:\N\ra\R^{n_x}:\quad H(\shiftop)[x]+L(\shiftop)[z]=0\big\},
\end{align}
which is called the healthy behavior of the system. The goal of the detection process is essentially to determine whether or not the current observed trajectory of~$z$ belongs to $\mathcal M$. Formally speaking, the detection task is to design a proper system whose input is the known signal $z$ and the respective output, also referred to as the filter {\em residual}, is zero when $z\in\mathcal{M}$ while it is nonzero when the aggregated fault signal~$\fa + \EZo\fm$ is nonzero. 
In~\cite{2006Nyberg}, a linear time invariant (LTI) system, also known as a residual generator, is proposed via the use of an irreducible polynomial basis for the nullspace of $H(\shiftop)$, denoted by $N_H(\shiftop)$. It is shown that such a polynomial fully characterizes the system behavior~\eqref{eq:beh1} in the sense of
\begin{align}\label{eq:beh2}
	\mathcal{M}=&\big\{z:\N\ra\R^{n_z}~|~N_H(\shiftop)L(\shiftop)[z]=0\big\}.
\end{align}
{In order to design a residual generator for the system~\eqref{eq:DAE}, fulfilling the desired conditions of fault detection, it suffices to introduce a linear filter polynomial $N(\shiftop)$ which can be characterized through the following polynomial arguments:
\begin{subequations}\label{eq:contot}
	\begin{align}
		N(\shiftop)H(\shiftop)=&0,\label{eq:con1}\\
		N(\shiftop)F(\shiftop)\neq& 0.\label{eq:con2}
	\end{align}
\end{subequations}
The first condition~\eqref{eq:con1} is concerned with the rejection of the natural disturbances and the unknown states, while the second condition~\eqref{eq:con2} ensures a non-zero response of the residual generator when the fault is non-zero.
In the light of Assumption~\ref{assu:detectability}, we restrict our attention to a proper LTI {estimation} filter of the form:
\begin{align}
	r\Let a^{-1}(\shiftop)N(\shiftop)L(\shiftop)[z],\label{eq:filtertf}
\end{align}
where the polynomial row vector $N(\shiftop)$ fulfills the requirements~\eqref{eq:contot}, and the stable transfer function $a^{-1}(\shiftop)$ is intended to make the residual generator proper (i.e., the degree of $a(\shiftop)$ is not less than the degree of $N(\shiftop)L(\shiftop)$ and stable (i.e., all the zeros of the polynomials $a(\shiftop)$ reside inside in the unit circle).
Following the definition of the residual~\eqref{eq:filtertf} and the DAE model~\eqref{eq:DAE}, it holds that the mapping from the signals~$\fa, \fm$ to the residual $r$ can be described by
\begin{align}\label{eq:faulttf}
    r = \tf{T}\big[E(z)\fm + \fa\big], \quad \text{where} \quad \tf{T} \Let -\frac{N(\shiftop)F(\shiftop)}{a(\shiftop)}.
\end{align}
A typical approach to isolate multiple faults~($\fa$, $\fm$) from one another is to introduce all the faults but one as natural disturbances encoded in the signal~$d$. However, this technique clearly fails for the DAE systems of the form~\eqref{eq:DAE} since Assumption~\ref{assu:detectability} does no longer hold in that case. In fact, by virtue of~\eqref{eq:faulttf}, one can see that the residual $r$ is linearly dependent on both fault signals~$\fa, \fm$. Due to this linear dependency, the residual can only be sensitive to the aggregated fault signal~$\fa+\EZo\fm$ and it is not possible to isolate this combination by means of linear filters, {an important scenario which we define in this work as {\em dynamical inseparability}}. This is the central fault isolation challenge studied in this work.}

\subsection{Outline of the proposed approach}
As mentioned in the preceding section, the key challenge of fault isolation is to estimate the additive fault $\fa$ and multiplicative fault $\fm$ when their impact on the dynamics (i.e., the corresponding dynamic matrix $F(\shiftop)$) are linearly dependent. In this study, we aim to address this challenge by leveraging tools from the regression theory, a well-known concept from the Machine learning literature~\cite{VanGestel20045}. However, to integrate those tools in a dynamical system setting and  provide rigorous performance guarantees, it is required to view these tools from a system-theoretic perspective and treat them as a dynamical system. This is the main part of the focus in this~study.  

More specifically, our proposed ``{\em {estimation} filter}" comprises three blocks, see Figure~\ref{fig:schemthm1}. The first block is called ``{\em fault detection}" and its role is to estimate the aggregated signal~$\fa+\EZo\fm$. This is essentially adopted from the current literature on fault detection with a slight extension that the residual signal~$r$ is expected to estimate
the behavior of~$\fa+\EZo\fm$ (rather than only acknowledging the existence of a fault). We call the second block ``{\em fault isolation}" that aims at isolating and estimating the contribution of the additive fault signal~$\fa$ and the multiplicative one~$\fm$. This block is essentially a (nonlinear) regression operator that also receives an additional signal~$e$, a required regressor signal containing the information of $\EZo$. As we will discuss in detail later, the dynamics of the system (and as such the dynamics of $\EZo$) has a nontrivial impact on the performance of the fault isolation block. This effect motivates the inclusion of the third block, to which we refer as the ``{\em pre-filter}". With regards to the pre-filter, we consider two cases in which one is a trivial identity (i.e., $e = \EZo$), and the second case is a linear transfer function with the input~$\EZo$, aiming to compensate for the dynamical behavior between the true aggregated signal and the residual.

\begin{figure}[t!]
        \centering
            \includegraphics[trim={0 21.6cm 0 4cm}]{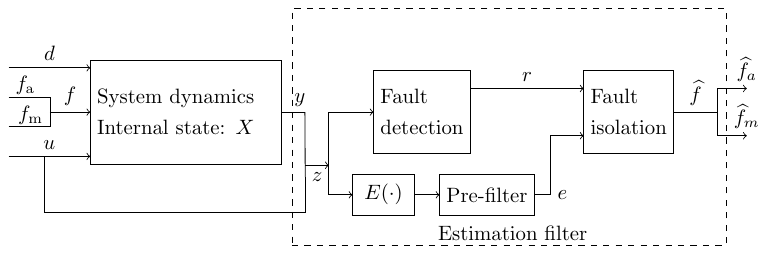}
        \caption{Block diagram of the proposed diagnosis filter.}
        \label{fig:schemthm1}
\end{figure}

\section{Estimation Filter Design: Main Results}\label{sec:results}

As sketched in Figure~\ref{fig:schemthm1}, the proposed estimation filter in this study comprises three blocks (i) fault detection, (ii) fault isolation, and (iii) pre-filter, which will be elaborated on in detail in this section. Here we only discuss the main results and their implications, and we will present the technical preliminaries and proofs in section~\ref{sec:technicalproofs}.

\subsection{Fault detection: linear residual generators}
The following lemma is a slight specialization of~\cite[Lemma~4.2]{Esfahani2016} that characterizes the class of linear residual generators with a desired asymptotic behavior. In this refined lemma, a steady-state condition on the residual is introduced. This serves as the basis for the detection block whose main objective is to detect and track (i.e., estimate) the aggregated fault signal~$\fa + \EZo\fm$. 

\begin{Lem}[LP characterization of fault detection]\label{lem:synthesis}
   Consider a polynomial row vector $N(\shiftop) = \sum_{i=0}^{d_N}N_i \shiftop^i$, and the system~\eqref{eq:DAE} with the model polynomial matrices
    \begin{alignat*}{2}
        H(\shiftop) = & \sum_{i=0}^{d_H}H_i \shiftop^i,\quad&  F(\shiftop)= \sum_{i=0}^{d_F}F_i \shiftop^i,\quad &  a(\shiftop) = \sum_{i=0}^{d_a}a_i\shiftop^{i},
    \end{alignat*}
    where $d_H,\:d_F,\:d_N,\:d_a$ denote the degree of matrices $H(\shiftop),\:F(\shiftop),\:N(\shiftop),\:a(\shiftop)$, respectively. Let us define the constant matrices
    \begin{align*}
        \widebar{N}\Let&\begin{bmatrix}N_{0}& N_{1}& \hdots&N_{d_N}\end{bmatrix}, \quad 
        &\widebar{a}\Let&\begin{bmatrix}a_0&a_1&\hdots&a_{d_a}\end{bmatrix}
        \\
        \widebar{H}\Let&
        \begin{bmatrix}
            H_{0}& H_1&\hdots&H_{d_H}&0&\hdots&0\\
            0&H_{0}&H_1&\hdots&H_{d_H}&0&\vdots\\
            \vdots&&\ddots&\ddots&&\ddots&0\\
            0&\hdots&0&H_{0}&H_1&\hdots&H_{d_H}\end{bmatrix},\quad
        &\widebar{F}\Let& \begin{bmatrix}
            F_{0}& F_1&\hdots&F_{d_F}&0&\hdots&0\\
            0&F_{0}&F_1&\hdots&F_{d_F}&0&\vdots\\
            \vdots&&\ddots&\ddots&&\ddots&0\\
            0&\hdots&0&F_{0}&F_1&\hdots&F_{d_F}\end{bmatrix}.
    \end{align*}
    Under Assumption~\ref{assu:detectability}, the linear program
    \begin{align}\label{eq:recastlemma}
        \begin{cases}
            \widebar{N}\,\widebar{H} = 0,\\
            \widebar{N}\, \widebar{F}\ind_{{d_N}\times{d_F}} = - \widebar{a}\ind_{{d_a}},
        \end{cases}
    \end{align}
    is feasible and any solution~$\widebar{N}$ is an admissible fault detector filter with a zero-steady state error from the aggregated fault to the residual , i.e., for any constant fault signals~$(\fa,\fm)$ and filter initial conditions, the residual~\eqref{eq:faulttf} fulfills $\lim_{t\to\infty}\fa+\EZo[(t)]\fm - r(t) = 0$. 
\end{Lem}
The proof is omitted as it is a straightforward adaptation from~\cite[Lemma 4.6]{8846711}.

\subsection{Fault Isolation: nonlinear regression} 
{Next, we present the design of the fault isolation block.} A central object of this part is the {\em regression operator}, a well-known scheme adopted from the machine learning literature~\cite{VanGestel20045}. This operator represents the fault isolation block whose domain and range spaces are discrete-time signals with appropriate dimensions. 

\begin{Def}[Regression operator]\label{def:operator}
	Given an integer~$n$ and scalar-valued signals $e$ and $r$, we define
	\begin{align}\label{eq:mphim}
	   \Phi_n[e,r](k) \Let \phi_n^{\dagger}[e](k)\,\vecsigk{r}(k), 
	   \qquad \text{where} \qquad \phact{e}\Let \big[\vecsig{e},~\ind_n\big]\in\mathbb{R}^{n \times 2},
	\end{align}	
	with the operator~$^\dagger$ as the pseudo-inverse.\footnote{$A^\dagger \Let (A\tr A)^{-1}A\tr$.}
\end{Def}

In the context of the fault estimation scheme in Figure~\ref{fig:schemthm1}, the output~$\Phi_n[e,r](k)$ of the nonlinear regression operation in Definition~\ref{def:operator} is, in fact, equal to the fault estimate $\widehat{f}$. The nonlinear regression operator in Definition~\ref{def:operator} enjoys certain regularity properties that are key for the results we will develop later. The following proposition provides input-output bounds of the regression operator. These bounds will be utilized later to develop a performance bound for the proposed estimation filter in a real-time and dynamic operational model. 

\begin{Prop}[Regression bounds]\label{prop:regression}
    Consider the regressor operator in Definition~\ref{def:operator}. For all discrete-time scalar-valued signals~$r, e$ and $y = [y^{(1)}, y^{(2)}]\tr$, at each time instant $k \in \N$ where $\varnk[]{e}{n}\neq0$ it holds that
    \begin{subequations}\label{eq:lem:bounds}
    \begin{align}
        \big\|\Phi_n[e,y^{(1)}+e\,y^{(2)}]-\mu_n[y]\big\|_2  & \leq \frac{\constant{n}(\vecsigk{e})}{\vark{e}}\Big(\vark{y^{(1)}}+\vark{y^{(2)}}\,\|\vecsigk{e}\|_\infty\Big),\label{eq:boundlem1}\\
        \big\|\Phi_n[e,r]\big\|_2  & \leq \frac{\constant{n}(\vecsigk{e})}{\sqrt{n}\vark{e}}\|\vecsigk{r}\|_2, \label{eq:boundlem2}
    \end{align}
    \end{subequations}
    where the constant is defined as~$\constant{n}(\vecsigk{e}) \Let \sqrt{\varnk[2]{e}{n} + \avek[2]{e} + 1}$. 
\end{Prop}

\begin{proof}
    The proof is provided in Section~\ref{subsec:pf-reg}.
\end{proof}

We emphasize that the bounds in~\eqref{eq:lem:bounds} hold for each time instant $k \in \N$, but to avoid clutter we drop the time-dependency of the signals (e.g., $\Phi_n[e,r]$ instead of $\Phi_n[e,r](k)$). We also note that the parameter~$\mathcal C$ only depends on the signal~$e$ (more precisely, on the last $n$ time instants of the signal~$e$ denoted by~$\vecsigk{e}$). In this view, the inequality~\eqref{eq:boundlem2} indeed represents an operator norm for the linear mapping $r \mapsto \Phi_n[e,r]$. Let us elaborate further on how the bounds as in~\eqref{eq:lem:bounds} are the first stepping-stones towards our main goal in this study. Measuring the ``aggregated" signal~$y^{(1)}+e\,y^{(2)}$, one can utilize the bound in~\eqref{eq:boundlem1} to bound the error on the  estimation of the average of the multivariate signal~$y = [y^{(1)}, y^{(2)}]\tr$ (i.e., $\mu_n[y]$) via the regression operator. It is worthwhile to note that when the signal~$y$ is constant, then $y = \mu_n[y]$ and $\vark{y^{(1)}} = \vark{y^{(2)}} = 0$, and that the estimation error reduces to zero provided that $\vark{e} \neq 0$. The second result~\eqref{eq:boundlem2} allows us to bound the output of the nonlinear regression operator given a bounded input, which can be viewed as a means to bound estimated faults given the dynamically filtered residual $r$ as an input. The bounds~\eqref{eq:lem:bounds} in fact offer a rigorous framework to treat the isolation block as a nonlinear dynamical system whose induced gain, and as such the boundedness of its output, is determined by~$\vark{e}$, the variance of signal~$e$ over a horizon with the length~$n$. 

\subsection{Pre-filter: dynamic compensator} 
In this section, we focus on the pre-filter block in Figure~\ref{fig:schemthm1}. Before presenting the main results of this paper, we first need to proceed with a basic preparatory lemma on the output bound of LTI systems. To improve the flow of the paper, we skip the technical proofs of the results in this section and defer them to Section~\ref{subsec:pf-thm}. 

\begin{Lem}[Zero steady-state LTI output bound]\label{lem:LTI}
    Consider a proper LTI system with the numerator~$ b(\shiftop) = \sum_{i=0}^{d} b_i \shiftop^i$, and the denominator~$a(\shiftop) = \prod_{i = 1}^{d} (\shiftop - p_i)$ where the poles are distinct and the dominant one  (i.e., the one closest to the unit circle) is $|p| < 1$. Suppose the steady-state gain of the filter is $0$ (i.e., $b(1) = 0$), the internal state (in the Jordan canonical form) is initiated at $X(0)$, and the input signal~$u(t)$ is $0$ until time $k_0$ and takes possibly nonzero values for $t \ge k_0$. Then, the output signal~$y(t)$ satisfies the bound  
    \begin{align*}
        \|\vecsigk{y}\|_2 \le \constant{0}\lVert X(0)\rVert_2|p|^{k{-n}} + \constant{1} \lVert\avenk[]{u}{k-k_0}\rVert_2|p|^{k{-n}-k_0} + \constant{2}{\sqrt{k-k_0}}\varnk[]{u}{k-k_0},
    \end{align*}
    where the constants~$\constant{0}, \constant{1}, \constant{2}$ are defined as
    \begin{align*}
        r_i = \frac{b(-p_i)}{\prod_{j \neq i}(p_j - p_i)}, \quad  
        \constant{0} = \sqrt{n \sum_{i=1}^{d} r_i^2}, \quad 
        \constant{1} = \frac{\,\sqrt{n\,d\sum_{i=1}^{d} r_i^2}}{1 - |p|}, \qquad \constant{2} = |b_d| + \sum_{i=1}^d \frac{|r_i|}{1 - |p_i|}\,. 
    \end{align*}
    \end{Lem}
        
\begin{proof}
The proof is provided in Section~\ref{subsec:pf-reg}.
\end{proof}
The statement of Lemma~\ref{lem:LTI} is rather classical and is not unexpected. However, we need such an assertion with explicit computable bounds for the main results of this study, which to our best knowledge does not exist in this form in the literature. 

We further propose two possible {designs} for the pre-filter, each of which comes along with certain pros and cons. The first, and simplest, design option is the static identity block. The next theorem presents a performance bound for this static pre-filter design.

\begin{Thm}[Performance bound: (I) static pre-filter]\label{theorem:main:I}
    Consider the system~\eqref{eq:DAE} and the fault estimation filter in Figure~\ref{fig:schemthm1} where the fault detection block is characterized by the linear program~\eqref{eq:recastlemma} and a denominator~$a(\shiftop)$ with distinct and real-valued poles. {The fault isolation block is the regression operator in~\eqref{eq:mphim} with the horizon $n$. Suppose the pre-filter block is identity (i.e., $e = E(z)$), and the fault signal starts at time~$k_0$.} Then, at each time instant~$k \in\N$ we have
    \begin{subequations}
	\label{eq:thm:global I}
	\begin{align}
		\label{eq:thm:perf I}
		\left\lVert\widehat{f}-\avek{f}\right\rVert_{2} \le  
		\frac{1}{\vark{e}}\Big( \alpha_0 |p|^{k  - k_0} + \alpha_1 \varnk[]{\fa}{k - k_0} + \alpha_2 \varnk[]{\fm}{k - k_0} + \alpha_3\Big), 
	\end{align}
	where the constant~$p\in\mathbb{R}$ is the dominant pole of the denominator~$a(\shiftop)$ and the involved constants are defined as
		\begin{align}
	    \alpha_0 & = 
		    \constant{1}\frac{\constant{n}(\vecsigk{e})}{\sqrt{n}} \Big(\lvert\avenk[]{\fa}{k-k_0}\rvert + \lvert\avenk[]{e\fm}{k-k_0}\rvert\Big),\label{eq:thm:perf I:1} \\
		    \alpha_1 & = \constant{2} \constant{n}(\vecsigk{e}){\sqrt{\frac{k-k_0}{n}}},\\
		    \alpha_2 & = \constant{2}\constant{n}(\vecsigk{e}){\sqrt{\frac{k-k_0}{n}}}\left({\sqrt{k-k_0}}\varnk[]{\EZ}{k-k_0} + \lvert\avenk[]{\EZ}{k-k_0}\rvert\right),\\
		    \alpha_3 & = \constant{n}(\vecsigk{e}) \Big( \vark{\fa} + \vark{\fm} \|\vecsigk{e}\|_\infty + \constant{2}{\sqrt{\frac{k-k_0}{n}}} \lvert\avenk[]{\fm}{k-k_0}\rvert\varnk[]{\EZ}{k-k_0}\Big).
		\end{align}
    \end{subequations}
    in which $\constant{n}(\vecsigk{e})$ is defined in Proposition~\ref{prop:regression} and the constants $\constant{1}, \constant{2}$ are defined in Lemma~\ref{lem:LTI}.
\end{Thm}

\begin{proof}
    The proof is provided in Section~\ref{sec:technicalproofs}.
 \end{proof}

By virtue of Theorem~\ref{theorem:main:I}, one can inspect how different aspects of the proposed design contribute to the fault estimation error. The most critical term is $\vark{e}$ in the denominator of the right-hand side of \eqref{eq:thm:perf I}. This challenging dependency is, however, an inherent limitation of the desired isolation task. In fact, one can show that when the signal~$E(z)$ is constant (i.e., $\vark{e} \equiv 0$), separation of the two faults~$(\fa, \fm)$ is even theoretically impossible. To reinforce this statement, consider the case $\vark{e} \equiv 0$ with arbitrary faults $\fa$ and $\fm$. It can be observed that the regression operator~\eqref{eq:mphim} contains the inverse of a degenerate component $\phactk{e}^\intercal\phactk{e}$ which, by definition, does not exist in such a case. This result shows that the signal $e$, over horizon $n$, does then not span the behavior of the aggregated fault over the same horizon, a concept close to the well-known {\em persistence of excitation} phenomena for LTI systems~\cite{WILLEMS2005325}.
The term~$\alpha_0$ in~\eqref{eq:thm:perf I:1} reflects the contribution of the average behavior of fault signals. In~\eqref{eq:thm:perf I:1} it can be seen that this term diminishes exponentially fast after the start of the fault signal due to its proportionality with the exponentially decaying term containing the dominant pole $|p|$ of the stable denominator $a(\mathfrak{q})$. In this light, we can deduce that the impact of these average behaviors on the performance is in fact negligible. The terms concerning $\alpha_1$ and $\alpha_2$ in~\eqref{eq:thm:global I} are mainly influenced by the variance of the fault since the beginning of the fault. The contribution of these variances in combination with the dynamics constants $\constant{1}, \constant{2}$ from Lemma~\ref{lem:LTI} is also an inevitable factor in the estimation error, since the regression model in Definition~\ref{def:operator} assumes constant contributions of the faults $\fa$ and $\fm$ appearing through the transfer function~\eqref{eq:faulttf} in the residual $r$ over a horizon $n$. Finally, the last term involving~$\alpha_3$ is a critical and potentially persistent source of error. In particular, the variance signal~$\varnk[]{\EZ}{k-k_0}$ introduces a non-zero estimation error even in the case of constant fault signals. The next corollary highlights this effect. 

\begin{Cor}[Constant faults: part I]\label{cor:constant faults:I}
    Consider the system and the estimation filter as in Theorem~\ref{theorem:main:I}. Suppose the fault signals are constant~$f = (\bar{\fa}, \bar{f}_m)$, starting from the time~$k_0$.
    Then, for any time instant $k \ge k_0 + n$ we have 
    \begin{align}\label{eq:cor:perf I}
		\left\lVert\widehat{f}-f\right\rVert_{2} \le & 
		\frac{\constant{n}(\vecsigk{e})}{\sqrt{n}\vark{e}}\Big( \constant{1}\big( |\bar{\fa}| + |\bar{f}_m| \avenk[]{e}{k - k_0}\big) |p|^{k {-n} - k_0} + {\constant{2}\sqrt{k-k_0}}|\bar{f}_m| \varnk[]{e}{k - k_0} \Big). 
	\end{align}
	\begin{proof}
        The proof is a direct application of Theorem~\ref{theorem:main:I}. Under the assumption that the fault signals are constants after time~$k \ge k_0$, we know that $\varnk[]{\fa}{k-k_0} = \varnk[]{\fm}{k-k_0} = 0$. Moreover, assuming further that $k \ge n + k_0$, we can also conclude that $\varnk[]{\fa}{n}(k) = \varnk[]{\fm}{n}(k) = 0$. In addition, the average terms of the signal reduces to  $\avenk[]{\fa}{k-k_0} = \bar{f}_a$ and $\avenk[]{\fm}{k-k_0} = \bar{f}_m$. Substituting these quantities in the bound\eqref{eq:thm:global I} concludes~\eqref{eq:cor:perf I}.
	\end{proof}
\end{Cor}

As noted above, the variance of the signal~$e$ is a persistent factor contributing to the performance bound, which is captured by the last term on the right-hand side of the inequality~\eqref{eq:cor:perf I}. This is somehow expected due to the causality effect of the system dynamics. More specifically, the residual~$r$, the output of the fault detection block, opts to follow the aggregated fault signal~$\fa + E(z)\fm$ but it relies on the dynamics $\tf{T}(\shiftop)$ (cf., \eqref{eq:faulttf}). However, when the pre-filter is set to identity (i.e., $e = E(z)$), the information of the signal is provided instantly for the isolation block ({due to the} static identity pre-filter), rendering some persistent potential error proportional to~$\varnk[]{e}{k - k_0}$. This error exists because the fault isolation block assumes a static mapping $e\mapsto r$ for the static pre-filter case, whereas this mapping is inherently dynamic due to the dynamics of the system and the fault detection block~\eqref{eq:faulttf}. This dynamical misalignment in the fault isolation block manifests itself in the estimation error, even for constant faults, as shown in~\eqref{eq:cor:perf I}.  Next, we aim to address this issue by filtering the information of the signal~$E(z)$ through the same dynamics that the residual of detection filter experiences. This {novel} viewpoint brings us to the second choice of pre-filter next. 

\begin{Thm}[Performance bound: (II) dynamic pre-filter] \label{theorem:main:II}
     Consider the system~\eqref{eq:DAE} and the fault estimation filter in Figure~\ref{fig:schemthm1} where the fault detection block is characterized by the linear program~\eqref{eq:recastlemma} and a denominator~$a(\shiftop)$ with distinct and real-valued poles. {The fault isolation block is the regression operator in~\eqref{eq:mphim} with the horizon $n$.} Suppose the pre-filter block is the linear system~$\tf{T}$ as defined in~\eqref{eq:faulttf} (i.e., $e = \tf{T}[E(z)]$) with the internal states denoted by~$X_p$. If the fault signal starts at time~$k_0$, then at each time instant~$k \in\N$ we have
    \begin{subequations}
    \label{eq:thm:global II}
        \begin{align}
        \label{eq:FIfiltpf}
		    \left\lVert\widehat{f}-\avek{f}\right\rVert_{2} \le  
		    \frac{1}{\vark{e}}\Big( \beta_0 |p|^{k- k_0} + \beta_1 \varnk[]{\fa}{k - k_0} + \beta_2 \varnk[]{\fm}{k - k_0} + \beta_3\Big),
	    \end{align}
	    where the constant~$p\in\mathbb{R}$ is the dominant pole of the denominator~$a(\shiftop)$ and the involved constants are defined as
	    \begin{align}
	        \beta_0 & = \frac{\constant{n}(\vecsigk{e})}{\sqrt{n}}  \Big(\constant{1}\left(\lvert\avenk[]{\fa}{n}\rvert +\lvert \avenk[]{E(z)\fm }{k-k_0} - \avenk[]{E(z)}{k-k_0}\avenk[]{\fm}{n}\rvert\right) +   \constant{0}\lvert\avenk[]{\fm}{n}\rvert\lVert X_{p}(k-k_0)\rVert_2\Big), \label{beta_0}\\
	        \beta_1 & =\constant{2}\constant{n}(\vecsigk{e}){\sqrt{\frac{k-k_0}{n}}},\\
	        \beta_2 & = \constant{2}\constant{n}(\vecsigk{e}){\sqrt{\frac{k-k_0}{n}}}\left({\sqrt{k-k_0}}\varnk[]{\EZ}{k-k_0}+\lvert\avenk[]{\EZ}{k-k_0}\rvert\right), \label{beta_2}\\
	        \beta_3 & = \constant{n}(\vecsigk{e}) \Big( \vark{\fa} + \vark{\fm}\left( \|\vecsigk{e}\|_\infty + \lVert\vecsigk{e} - E(\vecsigk{z})\rVert_\infty\right) \\
	        \nonumber & \qquad \qquad \qquad + \constant{2} \sqrt{\frac{k-k_0}{n}} \big | \avenk[]{\fm}{k-k_0} - \avenk[]{\fm}{n} \big| \varnk[]{E(z)}{k-k_0} \Big).
	    \end{align}
	\end{subequations}	
	in which $\constant{n}(\vecsigk{e})$ is defined in Proposition~\ref{prop:regression}, the constants $\constant{0}, \constant{1}, \constant{2}$ are defined in Lemma~\ref{lem:LTI}, and the vector-valued signal~$E(\vecsigk{z})$ is understood as the evaluation of the function~$E(\cdot)$ on each of the elements of the vector~$\vecsigk{z}$.
\end{Thm}

\begin{proof}
    The proof is provided in Section~\ref{sec:technicalproofs}.
\end{proof}

In a comparison with Theorem~\ref{theorem:main:I}, one can see that the main difference {in the fault estimation error bound} appears in the last coefficient of the error bounds~(cf, $\alpha_3$ in \eqref{eq:thm:perf I} and $\beta_3$ in \eqref{eq:FIfiltpf}). In particular, the idea of an appropriate dynamic pre-filter allows us to shift the contribution of the variance signal~$\varnk[]{e}{k - k_0}$ to the third term related to~$\beta_2$ in~\eqref{beta_2}, which is multiplied by the variance of the multiplicative fault~$\varnk[]{\fm}{k - k_0}$. This shift has a significant impact on the performance when the fault signals are constant during the activation time (i.e., $k \ge k_0$). Before proceeding with the simplification of the result, in this case, let us note that the dynamic pre-filter does not necessarily outperform the static one proposed by Theorem~\ref{theorem:main:I} due to the difference in the term $\vark{e}$. Indeed, the filtered signal $\tf{T}[E(z)]$ may have a lower variance, which has a negative impact on the performance bounds. 

\begin{Cor}[Constant faults: part II]\label{cor:constant faults:II}
     Consider the system and the estimation filter as in Theorem~\ref{theorem:main:II}. Suppose the fault signals are constant~$(\bar{\fa}, \bar{f}_m)$ starting at time~$k_0$. Then, for any time instant $k \ge k_0 + n$ 
    \begin{align}\label{eq:cor:perf II}
        \lVert\widehat{f}-f\rVert_2 \le \frac{\constant{n}(\vecsigk{e})}{\sqrt{n}\vark{\fEZ}} \Big(\constant{1}\lvert\bar{\fa}\rvert + \constant{0}\lvert\bar{f}_m\rvert\lVert X_{p}(k_0)\rVert_2
        \Big)
        |p|^{k-k_0}.
    \end{align}
	\begin{proof}
        In parallel to Corollary~\ref{cor:constant faults:I}, the proof is a direct application of Theorem~\ref{eq:thm:global II} when the fault signals are constants after time~$k \ge n + k_0$, and as such $\varnk[]{\fa}{k-k_0} = \varnk[]{\fm}{k-k_0} = 0$,  $\varnk[]{\fa}{n}(k) = \varnk[]{\fm}{n}(k) = 0$, $\avenk[]{\fa}{k-k_0} = \bar{f}_a$ and $\avenk[]{\fm}{k-k_0} = \bar{f}_m$. Besides, we also note that the term~$\avenk[]{\fm}{k-k_0} - \avenk[]{\fm}{n} = 0$ vanishes as well. Substituting these quantities in the bound~\eqref{eq:thm:global II} concludes~\eqref{eq:cor:perf II}.
	\end{proof}
\end{Cor}

In the case of constant faults, Corollary~\ref{cor:constant faults:II} indicates that the fault estimation error goes to zero exponentially fast if the filtered signal~$e$ behaves ``nicely" (i.e., $\vark{e}$ is uniformly away from zero). In comparison with the assertion of Corollary~\ref{cor:constant faults:I}, this outcome evidently highlights the role of the dynamic pre-filter on the estimation performance.

The following remark provides insight into the computational complexity of the used fault estimation methods.

\begin{Rem}[Computational complexity]
Given the fact that the optimal fault-detection filter~\eqref{eq:filtertf} and pre-filter~\eqref{eq:faulttf} are computed offline, the computational complexity of the real-time running algorithm, for both Theorem~\ref{theorem:main:I} and ~\ref{theorem:main:II}, is governed by the fault isolation block. Due to the regression operation, as defined in Definition~\ref{def:operator}, this method will have a time computational complexity of $\mathcal{O}(4n+8)$, where $n$ represents the prediction horizon of the fault isolation filter. 
\end{Rem}

Let us close this section with a brief summary of the results. In this section, a general estimation architecture has been proposed. System theoretic bounds for the regression operator (Definition~\ref{def:operator}) and the LTI bound (Lemma~\ref{lem:LTI}) have been used for the construction of guaranteed performance bounds for two pre-filter variants within this estimation architecture. The insights gained from the first pre-filter variant (i.e., the identity block in Theorem~\ref{theorem:main:I}) and its behavior for constant faults (Corollary~\ref{cor:constant faults:I}), have been leveraged to propose a second design variant (i.e., the dynamic pre-filter in Theorem~\ref{theorem:main:II}), which has been proven to have an exponentially decaying performance bound for constant faults (Corollary~\ref{cor:constant faults:II}).

\section{Technical Preliminaries and Proofs of Main Results}\label{sec:technicalproofs}
This section presents the technical proofs of the theoretical results in Section~\ref{sec:results}. To this end, we will also provide some preliminary results to facilitate the proof of the main results. Before proceeding with the proofs of the main theorems, we need first to provide a useful additional lemma concerning the variance of the product of two signals. The results of this section will later facilitate the proof of the main theorems in the subsequent section.

\subsection{Proofs of the regression and LTI output bounds}\label{subsec:pf-reg}
\begin{proof}[Proof of Proposition~\ref{prop:regression}]
    From \eqref{eq:mphim} recall that the matrix~$\Phactpsk{e}$ is the pseudo-inverse of $\phactk{e}$, and using the spectral norm properties we have
    \begin{align}
        \lVert\Phactpsk{e}\rVert_2 & =  \overline{\sigma}(\Phactpsk{e}) = \sqrt{\overline{\lambda}\left(\Phactpstk{e}\Phactpsk{e}\right)} = \sqrt{\overline{\lambda}\left(\left(\phacttk{e}\phactk{e}\right)^{-1}\phacttk{e}\phactk{e}\left(\phacttk{e}\phactk{e}\right)^{-1}\right)},\nonumber\\
        & = \sqrt{\overline{\lambda}\left(\left(\phacttk{e}\phactk{e}\right)^{-1}\right)} = \sqrt{\left(\underline{\lambda}\left(\phacttk{e}\phactk{e}\right)\right)^{-1}},
        \label{eq:spectralnormproof33}
    \end{align}
    where $\overline{\lambda}$ and $\underline{\lambda}$ represent the maximum and minimum eigenvalues, respectively. Based on Definition~\ref{def:operator}, the matrix $\phacttk{e}\phactk{e}\in\mathbb{R}^{2\times 2}$ can be expanded as
    \begin{align*}
        \phacttk{e}\phactk{e} = \begin{bmatrix}\sum_{i=0}^{n-1}(e(k-i))^2&\sum_{i=0}^{n-1}e(k-i)\\\sum_{i=0}^{n-1}e(k-i)&n\end{bmatrix}  = n\begin{bmatrix}\vark[2]{e}+\avek[2]{e}&\avek{e}\\\avek{e}&1\end{bmatrix}.
    \end{align*}
    Following some straightforward algebraic computations, one can find the minimum eigenvalue of the above matrix, yielding the bound of~\eqref{eq:spectralnormproof33} as
    \begin{align*}
         \lVert\Phactpsk{e}\rVert_2 = \sqrt{\frac{2}{n (a - \sqrt{a^2 - b})}} = \sqrt{\frac{2(a + \sqrt{a^2 - b})}{nb}}, 
         \qquad a \Let 1+\vark[2]{e}+\avek[2]{e} = , \quad b \Let 4\vark[2]{e}.
    \end{align*}
    Using the trivial bound $\sqrt{a^2 - b} \le |a|$, we arrive at the upper bound 
    \begin{align}\label{eq:lem33reg}
        \lVert\Phactpsk{e}\rVert_2 \leq \frac{\sqrt{1+\vark[2]{e}+\avek[2]{e}}}{\sqrt{n}\,\vark[]{e}} = \frac{\constant{n}(\vecsigk{e})}{\sqrt{n}\,\vark[]{e}}.
    \end{align}
    Given the upper bound above and following Definition~\ref{def:operator}, we can rewrite the regression operator as
    \begin{align}
    & \Phactk{e,y^{(1)} + e\,y^{(2)}} =  \Phactpsk{e}\left(\vecsigk{y}^{(1)}+\vecsigk{e}\circ\vecsigk{y}^{(2)}\right) \nonumber \\
    \label{eq:proof11}
    & \qquad = \Phactpsk{e}\Big( \vecsigk{y}^{(1)} + \vecsigk{e} \circ \vecsigk{y}^{(2)} + \avek{y^{(1)}}\ind_n - \avek{y^{(1)}}\ind_n + \avek{y^{(2)}}\vecsigk{e} - \avek{y^{(2)}}\vecsigk{e} \Big).
    \end{align}
    where the operator $\circ$ is the element-wise product between the vectors with an equal dimension. 
    We note that when $\vark{e} > 0$, the matrix $\phactk{e}$ is full rank, and that its pseudo-inverse~$\Phactpsk{e}$ is a well-defined object. Moreover, by virtue of the basic definition of the regression operator in \eqref{eq:mphim}, it is not difficult to observe that 
     \begin{align}\label{eq:avgidentity}
        \Phactpsk{e}\Big(\avek{y^{(1)}}\ind_n + \avek{y^{(2)}}  \vecsigk{e}\Big) = \Big[\avek{y^{(1)}}, \avek{y^{(2)}}\Big]\tr = \mu_n\big[[y^{(1)}, y^{(2)}]\tr\big] = \avek{y}.
    \end{align}   
    The interpretation of~\eqref{eq:avgidentity} is that if the coefficients of the signal~$e$ is constant over a horizon with length~$n$ (i.e., $\avek{y^{(1)}}, \avek{y^{(2)}}$), then the regression operator at each time instant retrieves these constants providing that the signal~$e$ has nonzero variance over the same horizon~(i.e., $\vark{e} > 0$). The observation~\eqref{eq:avgidentity} allows us to simplify the relation~\eqref{eq:proof11} to
    \begin{align*}
        \Phactk{e, y^{(1)} + e\circ y^{(2)}} = \Phactpsk{e} \left( \vecsigk{y}^{(1)} -\avek{y^{(1)}}\ind_n  + \vecsigk{e} \circ \left(\vecsigk{y}^{(2)} - \avek{y^{(2)}}\ind_n\right)\right) + \avek{y}.
    \end{align*}
    Bringing the term~$\ave{y}$ to the left-hand side and using the triangle inequality leads to
    \begin{align*}
        \lVert\Phactk{e,y^{(1)}+e\,y^{(2)}} - \avek{y} \rVert_2 & \leq \lVert \Phactpsk{e}\rVert_2 \Big\| \vecsigk{y}^{(1)} -\avek{y^{(1)}}\ind_n  + \vecsigk{e} \circ \Big(\vecsigk{y}^{(2)} - \avek{y^{(2)}}\ind_n \Big) \Big\|_2\\
         & \leq \lVert \Phactpsk{e} \rVert_2 \Big(\big\| \vecsigk{y}^{(1)} -\avek{y^{(1)}}\ind_n\big\|_2  + \Big\|\vecsigk{e} \circ \left(\vecsigk{y}^{(2)} - \avek{y^{(2)}}\ind_n\right) \Big\|_2\Big)\\
        & = \lVert\Phactpsk{e}\rVert_2\left(\sqrt{n}\left(\vark{y^{(1)}}+\vark{y^{(2)}}\lVert\vecsigk{e}\rVert_\infty\right)\right).
    \end{align*}
     Substituting the upper bound~\eqref{eq:lem33reg} in the above inequality concludes the first assertion~\eqref{eq:boundlem1}. The second claim~\eqref{eq:boundlem2} is a direct consequence of the upper bound~\eqref{eq:lem33reg} as follows:
     \begin{align*}
        \big\|\Phi_n[e,r]\big\|_2  =\lVert\Phactpsk{\EZ}\vecsigk{r}\rVert_2 \le \big\|\Phactpsk{\EZ}\big\|_2\big\|\vecsigk{r}\big\|_2 \le \frac{\constant{n}(\vecsigk{e})}{\sqrt{n}\vark{e}}\|\vecsigk{r}\|_2.
    \end{align*}
    This concludes the proof of Proposition~\ref{prop:regression}.
    \end{proof}

The second part of this section is concerned with the lemma about the LTI output bounds with zero steady-state.

\begin{proof}[Proof of Lemma~\ref{lem:LTI}]
        Given a time instant $k \in \N$, we decompose the time-varying control signal to two parts 
        \begin{align} \label{eq:u-decom}
        u(k) = \big(u(k) - \overline{u}\,\delta_{k_0}(k)\big) + \overline{u}\,\delta_{k_0}(k),
        \end{align}
        where the constant is set to  $\overline{u} = \aven[]{u}{k-k_0}$ and $\delta_{k_0}(k) = 1$ if $k \ge k_0$; otherwise $=0$. That is, the second input $\overline{u}\,\delta_{k_0}$ is seen as a constant input in the interval~$t \in [k_0, k]$. Leveraging the superposition property of the linear systems, one can view the outcome signal of the system as the response to the two different inputs in~\eqref{eq:u-decom}. In the first step, we focus on the response of the constant input signal~$\overline{u}\,\delta_{k_0}$ that starts from~$k_0$.  
        
        Let $(A,B,C,D)$ be a state-space realization of the LTI system. Note that by $0$-steady state property, we know that $D = - C(I - A)^{-1}B$. Utilizing the state-space description directly, we can rewrite the output signal at each time instant $k = \Delta k + k_0$, $\Delta k \ge 0$, as
        \begin{align*}
            y(k) & = \Big(C\big(\sum_{i=0}^{\Delta k-1} A^i\big)B + D\Big)\overline{u} + C A^k X(0) = \Big(C\big(\sum_{i=0}^{\Delta k-1} A^i -(I - A)^{-1}\big)B\Big)\overline{u} + C A^k X(0), 
        \end{align*}
        where in the second equality we use the $0$ steady state property of the transfer function. Since the matrix $A$ is stable (i.e., $|{\lambda}(A)| < 1$), we have $(I - A)^{-1} = \sum_{i \ge 0} A^i$. This equality allows us to simplify the above equality to
        \begin{align}
            y(k) & = -C\big(\sum_{i=\Delta k}^{\infty} A^i\big)B\overline{u} + CA^k X(0) \nonumber \\
            \label{y(k)}
            & = -CA^{\Delta k}\big(\sum_{i=0}^{\infty} A^i\big)B\overline{u} + CA^kX(0) = -CA^{\Delta k}\big((I-A)^{-1}B\overline{u}\big) + CA^kX(0).
        \end{align}
        We note that the LTI system can be equivalently written as
        \begin{align*}
            \frac{b(\shiftop)}{a(\shiftop)} = \frac{\sum_{i=0}^d b_i \shiftop^d}{\prod_{i = 1}^{d} (\shiftop - p_i)} = b_d + \sum_{i=1}^d \frac{r_i}{\shiftop + p_i}, 
        \end{align*}
        where $r_i$ is as defined in the lemma statement. We further focus on the particular choice of Jordan state-space canonical form in which $A = {\rm diag}\{[p_1, \dots, p_d]\}$, $B = [1, \dots,1]\tr$, $C = [r_1, \dots, r_n]$, and $D = b_d$. Thanks to diagonal structure of the matrix~$A$, we have $\|A \|_2 = |p|$. Therefore, using the triangle property of norms, the equality~\eqref{y(k)} yields the bound
        \begin{align*}
            \|y(k)\|_2 & \le \|C\|_2 \|A\|_2^{\Delta k} (1-\|A\|_2)^{-1}\|B\|_2|\overline{u}| + \|C\|_2 \|A\|_2^k\|X(0)\|_2 \\
            & = \|C\|_2|p|^{\Delta k} (1 - |p|)^{-1}\|B\|_2|\overline{u}| + \|C\|_2|p|^{k}\|X(0)\|_2.
        \end{align*}
        The above relation holds for every time instant~$k \in \N$ and $\Delta k = k - k_0 \ge 0$. Therefore, we can conclude that 
        \begin{align*}
            \|\vecsigk{y}\| \le \sqrt{n}\|C\|_2 |p|^{k-k_0 - n}(1 - |p|)^{-1}\|B\|_2|\overline{u}| + \|C\|_2 |p|^{k- n}\|X(0)\|_2,
        \end{align*}
        which in case of the special choices of $B$ and $C$ implies the desired constants~$\constant{0}$ and $\constant{1}$. We now focus on the response to the second part~$\big(u(k) - \overline{u}\delta_{k_0}(k)\big)$ in the input signal~\eqref{eq:u-decom}. Note that the $\mathcal{L}_2$-norm of~$u(\cdot) - \overline{u}\,\delta_{k_0}(\cdot)$ is indeed ${\sqrt{k-k_0}}\varnk[]{u}{k-k_0}$. In this view, the constant $\constant{2}$ can be upper bounded by the $\mathcal{H}_{\infty}$-norm of the transfer function; recall the $\mathcal H_{\infty}$-norm of a transfer function is the $\mathcal{L}_2$ induced norm of the transfer function as an operator. We then have 
        \begin{align*}
            \constant{2} = \Big\|\frac{a(\shiftop)}{b(\shiftop)}\Big\|_{\mathcal H_\infty} = \sup_{\Omega\in[0,2\pi]} \Big|\frac{a({\mathrm e}^{j\Omega})}{b({\mathrm e}^{j\Omega})}\Big| \le |b_d| + \sum_{i=1}^d \sup_{\Omega\in[0,2\pi]}\frac{|r_i|}{|\mathrm{e}^{j\Omega}-p_i|}
            \le |b_d| + \sum_{i=1}^d \frac{|r_i|}{1 - |p_i|},
        \end{align*}
        where the second equality follows from a classical result in control theory~\cite[Theorem 2]{126576}, the first inequality is due to the triangle inequality, and the second inequality is obtained as the maximum gain is attained from~$\Omega \in \{0,\pi\}$.
    \end{proof}
    
    \subsection{Bounds on the variance of product signals}\label{subsec:var-lemma}
    Before proceeding with the proof of the main theorems, we need first to provide a useful additional lemma concerning the variance of the product of two signals. The results of this section will later facilitate the proofs of the main theorem in the subsequent section. 
    
    \begin{Lem}[Variance of product signals]\label{lem:varbound}
    Consider the discrete-time signals $a$ and $b$ over a time-horizon~$n$. At each time instant, we have
    \begin{subequations}
    \begin{align}
        \lvert\varnk[2]{a+b}{n}-\varnk[2]{a}{n}-\varnk[2]{b}{n}\rvert & \le 2\min\left\{\lVert \vecsigk{a}\rVert_2\varnk[]{b}{n},\lVert \vecsigk{b}\rVert_2\varnk[]{a}{n}\right\},\label{eq:lem:varbound1}\\
        \varnk[]{ab}{n} & \leq \sqrt{n}\varnk[]{a}{n}\varnk[]{b}{n}+\lvert\avenk[]{a}{n}\rvert\varnk[]{b}{n}+\lvert\avenk[]{b}{n}\rvert\varnk[]{a}{n}.\label{eq:lem:varbound2}
    \end{align}
\end{subequations}
\end{Lem}

\begin{proof}
    We first note that since the desired assertion holds for everything time instant and horizon~$n$, the signals~$a, b$ can be viewed as an $n$-dimensional vectors. Let us start with the observation that
    \begin{align}    
        \nonumber 
        \varnk[2]{a+b}{n} & = \frac{1}{n}\lVert \vecsigk{a}+\vecsigk{b}\rVert_2^2-\avenk[2]{a+b}{n}
        = \frac{1}{n}\left(\lVert \vecsigk{a}\rVert_2^2+\lVert \vecsigk{b}\rVert_2^2\right) +2\avenk[]{ab}{n}-\avenk[2]{a}{n}-\avenk[2]{b}{n}-2\avenk[]{a}{n}\avenk[]{b}{n}\\ 
        \label{eq:var(a+b)}
        & = \varnk[2]{a}{n}+\varnk[2]{b}{n}+2\avenk[]{ab}{n}-2\avenk[]{a}{n}\avenk[]{b}{n}
        = \varnk[2]{a}{n}+\varnk[2]{b}{n}+2\avenk[]{a\left(b-\avenk[]{b}{n}\right)}{n}.
    \end{align}
    Note, that the expression $\avenk[]{a\left(b-\avenk[]{b}{n}\right)}{n}$ is equivalent to the inner product of the last $n$ time instants of the signals~$a$ and $\left(b-\avenk[]{b}{n}\right)$ multiplied by the constant~$n$. By virtue of the Cauchy-Schwarz inequality, for any signals~$v, w$ we have~$|\avek{v w}| \le n \|\vecsigk{v}\|_2 \|\vecsigk{w}\|_2$. This bound together with the equality~\eqref{eq:var(a+b)} yields
    \begin{align*}
        \varnk[2]{a}{n}+\varnk[2]{b}{n} - 2\lVert \vecsigk{a}\rVert_2\varnk[]{b}{n} \le \varnk[2]{a+b}{n} & \leq \varnk[2]{a}{n}+\varnk[2]{b}{n}+2\lVert \vecsigk{a}\rVert_2\varnk[]{b}{n},\\
        \Longleftrightarrow \quad  
        \lvert\varnk[2]{a+b}{n}-\varnk[2]{a}{n}-\varnk[2]{b}{n}\rvert & \leq 2\lVert \vecsigk{a}\rVert_2\varnk[]{b}{n}.
    \end{align*}
    Thanks to the symmetry between $a$ and $b$ in the above implication, a possibly less conservative upper-bound can be deduced:
    \begin{align*}
        \lvert\varnk[2]{a+b}{n}-\varnk[2]{a}{n}-\varn[2]{b}{n}\rvert \leq 2\min\big\{\lVert \vecsigk{a}\rVert_2\varnk[]{b}{n},\lVert \vecsigk{b}\rVert_2\varnk[]{a}{n}\big\}.
    \end{align*}
    This concludes the assertion~\eqref{eq:lem:varbound1}. To show~\eqref{eq:lem:varbound2}, we define two constants at the given time instant~$k$ as
    \begin{align}\label{eq:a-bar}
        \overline{a} \Let \ave{a}, \qquad \overline{b} \Let \ave{b}. 
    \end{align}
    We emphasize that we view~$\overline{a}, \overline{b}$ as two constants over the entire history, and that we possibly have $\avek{a}(k') \neq \overline{a}$ if $k' < k$. In other words, the vector~$\big[\avek{a}(k-n+1),\dots,\avek{a}(k)\big]\tr$ is {\em not} the same as $\overline{a}\,\ind_n$; these two vectors only agree on the very last component. Using the definitions~\eqref{eq:a-bar}, we have
    \begin{align*}
        \varnk[2]{ab}{n} & = V^2_n \big[ \left(a - \overline{a}+\overline{a}\right)\left(b-\overline{b}+\overline{b}\right) \big] 
        = V^2_n \Big[ \left(a-\overline{a}\right)\left(b-\overline{b}\right)+\overline{b}\left(a-\overline{a}\right)+\overline{a}\left(b-\overline{b}\right)+\overline{a}\overline{b} \Big]
    \end{align*}
    Since the variance operator is invariant under a constant offset, the above relation reduces to
    \begin{align}\label{eq:lem:varbound3}
        \varnk[2]{ab}{n} =  V^2_n \Big[ \left(a-\overline{a}\right)\left(b-\overline{b}\right)+\overline{b}\left(a-\overline{a}\right)+\overline{a}\left(b-\overline{b}\right) \Big]. 
    \end{align}
    To simplify the notation, let us introduce
    \begin{align}\label{v,w}
        v \Let \left(a-\overline{a}\right)\left(b-\overline{b}\right), \qquad w \Let \overline{b}\left(a-\overline{a}\right)+\overline{a}\left(b-\overline{b}\right).
    \end{align}
    Using the above definitions and the result from~\eqref{eq:lem:varbound1}, we can derive the following bound from the equality~\eqref{eq:lem:varbound3}:
    \begin{subequations}
    \begin{align}
        \label{eq:var(m1+m2):1}
        \varnk[2]{ab}{n} & = \varnk[2]{v+w}{n} \leq \varnk[2]{v}{n}+\varnk[2]{w}{n} + 2\min\left\{\lVert \vecsigk{v}\rVert_2\varnk[]{w}{n},\lVert \vecsigk{w}\rVert_2\varnk[]{v}{n}\right\} \\
        \label{eq:var(m1+m2):2}
        & \le \varnk[2]{v}{n}+\varnk[2]{w}{n}+2\varnk[]{v}{n}\varnk[]{w}{n} = \left(\varnk[]{v}{n}+\varnk[]{w}{n}\right)^2,
    \end{align}
    \end{subequations}
    where the inequality in~\eqref{eq:var(m1+m2):1} is a direction application of~\eqref{eq:lem:varbound1}, and the inequality in~\eqref{eq:var(m1+m2):2} follows from the fact that the signal~$w$ has a zero mean and thus $\|\vecsigk{w}\|_2 = \vark{w}$. 
    Substituting back the definitions of $v$ and $w$ from~\eqref{v,w} into the inequality~\eqref{eq:var(m1+m2):2} yields
    \begin{align*}
        \varnk[2]{ab}{n} & \leq \left(\sqrt{n}\varnk[]{a}{n}\varnk[]{b}{n} + \frac{1}{\sqrt{n}}\big\|\overline{b}\left(\vecsigk{a} - \overline{a}\ind_n\right)\big\|_2 + \frac{1}{\sqrt{n}}\big\|\overline{a}\left(\vecsigk{b} - \overline{b}\,\ind_n\right)\big\|_2\right)^2 \\
        & = \left(\sqrt{n}\varnk[]{a}{n}\varnk[]{b}{n}+\frac{|\overline{b}|}{\sqrt{n}}\big\|\left(\vecsigk{a} - \overline{a}\ind_n\right)\big\|_2 + \frac{|\overline{a}|}{\sqrt{n}}\big\|\left(\vecsigk{b} - \overline{b}\,\ind_n\right)\big\|_2\right)^,
    \end{align*}
    which concludes the desired assertion~\eqref{eq:lem:varbound2} since~${\sqrt{n}}\vark{a} = \big\|\left(\vecsigk{a} - \overline{a}\ind_n\right)\big\|_2$.
\end{proof}

\subsection{Proofs of the main theorems}\label{subsec:pf-thm}

\begin{proof}[Proof of Theorem~\ref{theorem:main:I}]
    Let us first introduce the shorthand notation
    \begin{align}\label{eq:G-delta}
        \tf{G} \Let \tf{T} - \tf{I}, \qquad \delta(k) \Let \fa + e\fm(k), \qquad e = E\big(z(k)\big),
    \end{align}
    where the transfer function~$\tf{T}$ is as defined in~\eqref{eq:faulttf} and $\tf{I}$ is the identity transfer function. Notice that in this part the pre-filter is the static gain identity, and that its output signal~$e$ is indeed the measurement signal~$E(z)$~(cf. Figure~\ref{fig:schemthm1}). 
    Based on the definition of the estimated fault and the regression operator in Definition~\ref{def:operator} we have
    \begin{align*} 
        \widehat{f} = \Phi_n[e, r] = \Phi_n[e, r - \delta] + \Phi_n[e, \delta] - \avek{f} + \avek{f}, 
    \end{align*}   
    where the second equality simply follows from the linearity of the regression operator in the second argument. Let moving the term~$\avek{f}$ to the left-hand side and taking the 2-norm on both sides of the above equality. Using the triangle inequality and the regression bounds from Proposition~\ref{prop:regression}, we arrive at 
    \begin{align}
        \|\widehat{f}-\avek{f}\|_2 & \le \|\Phi_n[e, r - \delta]\|_2 + \|\Phi_n[e, \delta] - \avek{f}\|_2 \nonumber \\
        \label{eq:f_bounds}
        & \le \frac{\constant{n}(\vecsigk{e})}{\sqrt{n}\vark{e}} \| \vecsigk{r} - \vecsigk{\delta}\|_2 + \frac{\constant{n}(\vecsigk{e})}{\vark{e}} \big(\vark{\fa} + \vark{\fm}\lVert\vecsigk{e}\rVert_\infty \big), 
    \end{align}
    where the first and second bounds in \eqref{eq:f_bounds} follow from \eqref{eq:boundlem2} and \eqref{eq:boundlem1}, respectively. 
    It then remains to bound the term~$\|\vecsigk{r}-\vecsigk{\delta}\|_2$ on the right-hand side of~\eqref{eq:f_bounds}. Following the definitions of the residual~$r$ in~\eqref{eq:faulttf}, and the signal~$\delta$ and the transfer function~$\tf{G}$ in~\eqref{eq:G-delta}, we have 
    \begin{align}
        r-\delta = \tf{T}\left[\fa + E(z) \fm\right]-\left(\fa+e\fm\right)\nonumber = \tf{G}\left[\fa\right]+\tf{G}\left[e\fm\right].\nonumber
    \end{align}
    Note that by construction the transfer function~$\tf{G}$ has a zero steady-state gain since the transfer function~$\tf{T}$ has $unit$ steady-state gain (see Lemma~\ref{lem:synthesis}). As such, we can apply Lemma~\ref{lem:LTI} to the right-hand side of the above relation. This leads to 
    \begin{align*}
        \lVert\vecsigk{r}-\vecsigk{\delta}\rVert_2 & \leq  \constant{1}\left(\lvert\avenk[]{\fa}{k-k_0}\rvert + \lvert\avenk[]{e\fm}{k-k_0}\rvert\right)\lvert p\rvert^{k-k_0} + \constant{2}{\sqrt{k-k_0}}\left(\varnk[]{\fa}{k-k_0} + \varnk[]{e\fm}{k-k_0}\right),\nonumber\\
        & \leq \constant{1}\left(\lvert\avenk[]{\fa}{k-k_0}\rvert + \lvert\avenk[]{e\fm}{k-k_0}\rvert\right)\lvert p \rvert^{k-k_0} +  \constant{2} {\sqrt{k-k_0}} \Big( \varnk[]{\fa}{k-k_0} \nonumber \\
        & \qquad \qquad + {\sqrt{k-k_0}}\varnk[]{\fm}{k-k_0}\varnk[]{e}{k-k_0} + \lvert\avenk[]{\fm}{k-k_0}\rvert\varnk[]{e}{k-k_0}+\lvert\avenk[]{e}{k-k_0}\rvert\varnk[]{\fm}{k-k_0} \Big),
    \end{align*}
    where in the last line we apply \eqref{eq:lem:varbound2} from Lemma~\ref{lem:varbound} to the variance of the product signals $\varnk[]{e\fm}{k-k_0}$. Substituting the above bound in~\eqref{eq:f_bounds} results in
    \begin{align*}
        \lVert\widehat{f}-\avek{f}\rVert_2 & \leq \frac{\constant{n}(\vecsigk{e})}{\sqrt{n}\vark{e}}\Big(\constant{1}\left(\lvert\avenk[]{\fa}{k-k_0}\rvert+\lvert\avenk[]{e\fm}{k-k_0}\rvert\right)\lvert p \rvert^{k-k_0}\nonumber \\ 
        & \qquad + \constant{2}{\sqrt{k-k_0}}\big(\varnk[]{\fa}{k-k_0} + {\sqrt{k-k_0}}\varnk[]{\fm}{k-k_0}\varnk[]{e}{k-k_0}\\
        & \qquad + \lvert\avenk[]{\fm}{k-k_0}\rvert\varnk[]{e}{k-k_0}+\lvert\avenk[]{e}{k-k_0}\rvert\varnk[]{\fm}{k-k_0}\big)\Big)\nonumber+\frac{\constant{n}(\vecsigk{e})}{\vark{e}}\left(\vark{\fa}+\vark{\fm}\lVert\vecsigk{e}\rVert_\infty\right).
    \end{align*}
    Finally, it suffices to factor out the right-hand side of the above inequality to the exponentially decaying term and the variance terms~$\varnk{\fa}{k-k_0}$, $\varnk{\fm}{k-k_0}$, as well as the remaining parts including $\varnk{\fa}{n}, \varnk{\fm}{n}, \varnk{e}{k-k_0}$. This concludes the desired assertion~\eqref{eq:thm:global I}.
\end{proof}

\begin{proof}[Proof of Theorem~\ref{theorem:main:I}]
The key difference between the setting of this theorem with Theorem~\ref{theorem:main:I} is the choice of pre-filter, and as such, the definition of the signal~$e$. Consider the same definitions of the transfer function~$\tf{G}$ and the signal~$\delta$ as in~\eqref{eq:G-delta} where the output of the pre-filter is defined as
\begin{align}
    e = \tf{T}\left[E(z)\right], \qquad \text{with the internal states~$X_p$.}
\end{align}
Note that the relation~\eqref{eq:f_bounds} still holds in the setting here as well. We then only need focus on the term $\lVert \vecsigk{r}-\vecsigk{\delta}\rVert_2$. In the rest of the proof, we fix the time instant~$k \in \N$ and define the average of the multiplicative fault~$\fm$ over the horizon $[k-n, k]$ as the constant denoted by
\begin{align}\label{f_bar}
    \overline{\fm} \Let \aven{\fm}{n}. 
\end{align}
Let us emphasize that we view the value $\overline{\fm}$ as constant over the entire time horizon prior to $k$; see also the definitions~\eqref{eq:a-bar} and the following paragraph there. We further introduce the shorthand notation of the step function
\begin{align*}
    \step{k_0}(k) \Let \left\{
        \begin{array}{ll}
              0\quad k< k_0\\
              1\quad k\geq k_0 \,.
        \end{array}\right.
\end{align*}

Following the above definitions, the signal~$r - \delta$ can be rewritten as
\begin{subequations}
\label{eq:r-d}
\begin{align}
        \nonumber 
        r - \delta & = \tf{T}\left[\fa + E(z)\fm\right] - \left(\fa+\tf{T}\left[E(z)\right]\fm\right) = \tf{G} \left[\fa\right] + \tf{T}\left[E(z)\fm\right] - \tf{T}\left[E(z) \right]\fm\\
        \nonumber 
        & = \tf{G}\left[\fa\right]+\tf{T}\left[E(z)\fm - \overline{\fm}E(z)\step{k_0}\right] +  \tf{T}\left[ \overline{\fm}E(z)\step{k_0} \right] - \tf{T}\left[E(z)\right]\fm \\
        \label{eq:r-d:2}
        & = \tf{G}\left[\fa\right]+\tf{T}\left[E(z)(\fm - \overline{\fm}\step{k_0})\right] +  \overline{\fm}\tf{T}\left[E(z)(\step{k_0} - 1) \right]\step{k_0} - \tf{T}\left[E(z)\right](\fm -\overline{\fm}\step{k_0})\\
        \label{eq:r-d:3} 
        & = \tf{G}\left[\fa\right] + \tf{G}\left[E(z)(\fm - \overline{\fm}\step{k_0})\right] +  \overline{\fm}\tf{T}\left[E(z)(\step{k_0} - 1) \right]\step{k_0} - \tf{G}\left[E(z)\right](\fm -\overline{\fm}\step{k_0}),
    \end{align}
    \end{subequations} 
where in \eqref{eq:r-d:2} we take the constant~$\overline{\fm}$ outside the linear transfer function~$\tf{T}$. We also note that to arrive at~\eqref{eq:r-d:2} we use the fact that~$\tf{T}\left[E(z)\step{k_0}\right](k) = 0$ for all time instants $k < k_0$, i.e., $\tf{T}\left[E(z)\step{k_0}\right] = \tf{T}\left[E(z)\step{k_0}\right]\step{k_0}$. Recall that $\tf{G} = \tf{T} - \tf{I}$ is a stable transfer function with zero steady-state gain. Also, note that for $k \ge k_0$ the third term~$\tf{T}\left[E(z)(\step{k_0} - 1) \right]$ in~\eqref{eq:r-d:3} is in fact the contribution of the internal states~$X_p(k_0)$ of the transfer function~$\tf{T}$ when the input signal is~$E(z)$ ($\step{k_0} - 1 = 0$ for all $k \ge 0$). Therefore, we can apply Lemma~\ref{lem:LTI} to each term on the right-hand side in~\eqref{eq:r-d:3} and obtain the bound
\begin{align}
        \lVert \vecsigk{r} - \vecsigk{\delta} \rVert_2 
        & \leq \constant{1}\lvert\overline{\fa}\rvert|p|^{k-k_0}+\constant{2}\sqrt{k-k_0}\varnk[]{\fa}{k-k_0} \label{eq:varproduct}\\
        & \qquad + \constant{1}\lvert\avenk[]{E(z)(\fm - \overline{\fm}\step{k_0})}{k-k_0}\rvert|p|^{k-k_0}+\constant{2}\sqrt{k-k_0}\varnk[]{E(z)(\fm - \overline{\fm}\step{k_0})}{k-k_0} \nonumber \\
        & \qquad + \lvert\widebar{f_m}\rvert\constant{0}\lVert X_p(k_0)\rVert_2|p|^{k-k_0} \nonumber\\
        & \qquad + \lVert \vecsigk{e}-E(\vecsigk{z})\rVert_\infty\sqrt{n}\varnk[]{\fm}{n}\nonumber
\end{align}
We first note that in \eqref{eq:varproduct} we can simplify the first term of the second line as $\avenk[]{E(z)(\fm - \overline{\fm}\step{k_0})}{k-k_0}=\avenk[]{E(z)\fm }{k-k_0}-\avenk[]{E(z)}{k-k_0}\overline{\fm}$. We further borrow the results of Lemma~\ref{lem:varbound} to bound the terms involving the product of two signals in~\eqref{eq:varproduct}. More specifically, we have
\begin{align}
    \varnk[]{E(z)(\fm-\overline{\fm}\step{k-k_0})}{k-k_0} & \le \sqrt{k-k_0}\varnk[]{E(z)}{k-k_0}\varnk[]{\fm}{k-k_0}\nonumber\\
    &\qquad + \lvert\avenk[]{\fm - \overline{\fm}\step{k_0}}{k-k_0}\rvert\varnk[]{E(z)}{k-k_0}\nonumber + \lvert\avenk[]{E(z)}{k-k_0}\rvert\varnk[]{\fm}{k-k_0} \\
    & = \sqrt{k-k_0}\varnk[]{E(z)}{k-k_0}\varnk[]{\fm}{k-k_0} \label{eq:products}\\
    &\qquad + \lvert\avenk[]{\fm}{k-k_0} - \overline{\fm}\step{k_0}\rvert\varnk[]{E(z)}{k-k_0} + \lvert\avenk[]{E(z)}{k-k_0}\rvert\varnk[]{\fm}{k-k_0} \nonumber. 
\end{align}
It now suffices to substitute the upper bounds~\eqref{eq:products} in \eqref{eq:varproduct}, and then invoke the resulting bound on~$\lVert \vecsigk{r} - \vecsigk{\delta} \rVert_2$ in~\eqref{eq:f_bounds}. Finally, it remains to factor out the right-hand side of the inequality to the exponentially decaying term, the variance terms~$\varnk{\fa}{k-k_0}$, $\varnk{\fm}{k-k_0}$, as well as the remaining parts including $\varnk{\fa}{n}$, $\varnk{\fm}{n}$, $\varnk{E(z)}{k-k_0}$. This concludes the desired assertion~\eqref{eq:thm:global II}.
\end{proof}


\section{Case Study: Lateral Control of Autonomous Vehicles}\label{sec:casestudy}
\begin{wrapfigure}{R}{0.4\textwidth}
\centering
\includegraphics[clip, trim=6cm 13.1cm 0cm 7.2cm, scale=0.7]{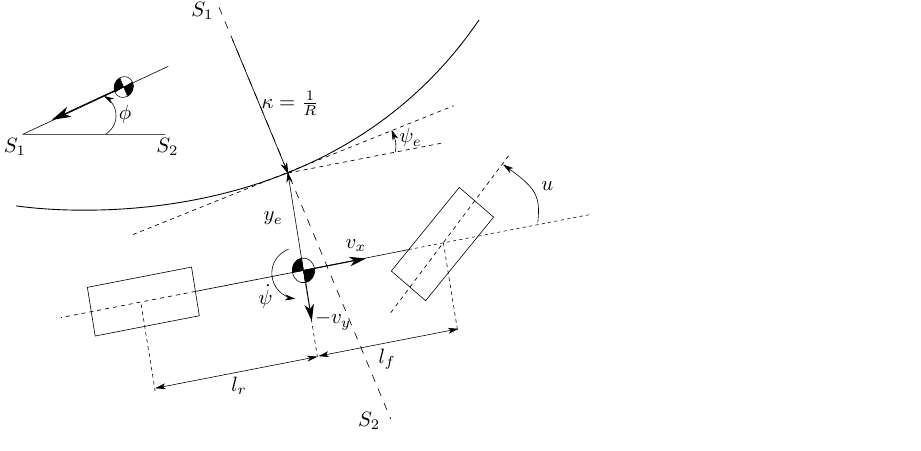}
\caption{Visual representation of the bicycle model.}
\label{fig:simmodel}
\end{wrapfigure}
In this section, the presented theory is illustrated using a fault isolation problem in the scope of the lateral control of autonomous vehicles. In this scope, the linear single-track vehicle model is used as a benchmark example~\cite[Equation 1]{Schmeitz2017}. The model represents a linearization of the lateral dynamics of an automated vehicle. As more driver tasks are alleviated in the context of automated driving, the responsibility of passenger safety is expected to shift towards the automated driving software. In this context, fault detection and isolation are increasingly important in the automotive industry. 
Namely, robustness for failures from both known and unknown sources remains a challenging topic in this domain and requires extensive research and rigorous experimental evaluation~\cite{GUANETTI201818}. The outline of this section is as follows: first, a model description is given in Section~\ref{sec:synthesis}, simulation results will be shown and a conclusion will be drawn regarding filter performance and bound analysis in Section~\ref{sec:simresults}. Finally, a sensitivity analysis is performed to explore the degrees of freedom in the filter design in Section~\ref{sec:sensitivity}.

\subsection{Problem description}\label{sec:synthesis}
The proposed linear bicycle model is depicted in Figure~\ref{fig:simmodel}. Here, the state of the vehicle is chosen as $X=[v_y,\:\dot{\psi},\:y_e,\:\psi_e]\tr$ where $v_y$ represents the lateral velocity, $\dot{\psi}$ represents the yaw-rate, $y_e$ represents the lateral error from the lane centre and $\psi_e$ represents the heading error from the lane centre. The disturbance vector is chosen as $d=\left[\sin(\phi),\:\kappa\right]$ where $\phi$ represents the banking angle of the road (as depicted by cross-section $S_1-S_2$ in Figure~\ref{fig:simmodel}) and $\kappa$ represents the curvature of the road. The input $u$ represents the steering wheel angle of the front wheels of the vehicle. In this case study, we consider additive and multiplicative faults acting on the steering input signal $u$. These faults could realistically occur as an offset in the steering column, $\fa$, or a loss of actuator efficiency, $\fm$. These faults may result in unexpected transient and steady-state tracking errors and hence result in dangerous situations for the vehicle passengers if not detected and handled properly. The model description with its states, disturbances, faults and input can be written in the form of a continuous-time, linear differential equation as follows:
\begin{align}\label{eq:statespaceODE}
\left\{\begin{array}{lc}
\dot{X}(t)=\bar{A}X(t)+\bar{B}_uu(t)+\bar{B}_f({\fm(t)u(t)+\fa(t)})+\bar{B}_dd(t),\\Y(t)=CX,\end{array}\right.
\end{align}

where the system matrices $\bar{A},\:\bar{B}_u,\:\bar{B}_d\:,C$ can be written as follows~\cite[Equation 1]{Schmeitz2017}:
    \begin{align*}
    \footnotesize
    \setlength{\arraycolsep}{2.5pt}
    \medmuskip = 1mu 
        \bar{A} =  \begin{bmatrix}\frac{C_f+C_r}{v_xm}&\frac{l_fC_f-l_rC_r}{v_xm}&0&0\\
    \frac{l_fC_f-l_rC_r}{v_xI}&\frac{l_f^2C_f+l_r^2C_r}{v_xI}&0&0\\-1&0&0&v_x\\0&-1&0&0\end{bmatrix}, \quad 
        \bar{B}_u = \begin{bmatrix}-\frac{C_f}{m}\\-\frac{l_fC_f}{I}\\0\\0\end{bmatrix}, \quad 
        \bar{B_d} = \begin{bmatrix}g&0\\0&0\\0&0\\0&v_x\end{bmatrix},\quad C = \begin{bmatrix}0&1&0&0\\0&0&1&0\\0&0&0&1\end{bmatrix},
    \end{align*}
where $C_f=1.50\cdot 10^5\:\rm N\cdot rad^{-1}$ and $C_r=1.10\cdot 10^5\:\rm N\cdot rad^{-1}$ represent the lateral cornering stiffnesses of the front and rear tyres, respectively, $l_f=1.3\:\rm m$ and $l_r=1.7\:\rm m$ represent the distances from the front and rear axle to the center of gravity, respectively, $v_x=19\:\rm m\cdot s^{-1}$ represents the longitudinal velocity, $m=1500\:\rm kg$ represents the total mass of the vehicle, $I=2600\:\rm kg\cdot m^2$ represents the moment of inertia around the vertical axis of the vehicle and finally $g=9.81\:\rm m\cdot s^{-2}$ represents the gravitational acceleration.     
To fit the discrete-time model setting employed in this paper, we first exactly discretize the dynamical system using a classical control theoretical result~{\cite[Section 2]{Yamamoto1996166}}. The resulting discrete-time state-space matrices can be written as follows:
\begin{align}
	\left\{\begin{array}{lc}
    	A=e^{\bar{A}h},&B_u=\int_{0}^he^{\bar{A}s}\bar{B_u}ds,\\B_d = \int_{0}^he^{\bar{A}s}\bar{B}_dds,&B_f = \int_{0}^he^{\bar{A}s}\bar{B}_fds
    	,\label{eq:matrices}
	\end{array}\right.
\end{align}
where $h$ is the sampling interval and is chosen as $h=0.01s$ and $f=\fm u+\fa$ represents the aggregated fault signal. Note, that the discretized system matrices~\eqref{eq:matrices} can be written in the DAE framework by virtue of Fact~\ref{fact:odedae} when setting $K_X=0$ and $K_Y=0$.

\subsection{Simulation results}\label{sec:simresults}
For the synthesis of the fault detection filter, the degree of the filter $N(\shiftop)$ is set to $d_{N}=3$. The denominator $a(\shiftop)$ of the fault detection filter is set to $a(\shiftop)=(\shiftop+0.85)(\shiftop+0.59)(\shiftop+0.58)$. Note, that distinct poles are chosen to be able to verify the performance bound of the estimation error using the results of Lemma~\ref{lem:LTI}. The filter $N(\shiftop)$ can be found by solving the linear program in~\eqref{eq:recastlemma}. For synthesis of the fault isolation filter, the time horizon $n$ is initially chosen as $n=10$. This forms a basis for the fault diagnosis filter.\\
The input signal $u$ to the system dynamics is selected to be a sinusoidal signal with an amplitude of $2.3\cdot 10^{-3}$ radians at a frequency of $0.3$Hz. The frequency content of the input signal is inspired by experimental data of an automated vehicle, driving in-lane using a PD-type controller, while being excited by natural disturbances (e.g., profile of the road). For this simulation study, the additive fault and the multiplicative fault are selected as {\em incipient} fault functions
\begin{align*}
    \fa(k)\Let\left\{\begin{array}{@{}c}0\\2.5\cdot 10^{-4}\frac{\pi}{180}k\\0.1\frac{\pi}{180}\end{array}\begin{array}{@{}c}0\leq k<850\\850\leq k< 1250\\k\geq1250\end{array}\right.,\qquad\fm(k)\Let\left\{\begin{array}{c}-\frac{0.05}{100}k\quad\\-0.2\end{array}\begin{array}{c}0\leq k<400\\k\geq400\end{array}\right..
\end{align*}
These fault functions are chosen to realistically correspond to an experimental setting. The additive fault $\fa$ usually appears as a low magnitude fault as it could be caused by a small mechanical misalignment in the steering rack. The multiplicative fault $\fm$ is chosen to represent a degradation of the efficiency of the steering actuator. This is a critical case in automated driving as a continuously degrading actuator could result in a decision to park towards the side of the road. The practical value of the faults and the system described shows the particular value of this simulation setup. Using the described model and its inputs, the two faults are injected and the estimation error and performance bounds from Theorem~\ref{theorem:main:I} and~\ref{theorem:main:II} are simulated/evaluated as shown in Figure~\ref{fig:reso1}.\\
\begin{figure}
     \centering
     	\subfloat[Additive fault $\fa$]{\includegraphics[scale=0.8, trim={2cm 12.5cm 0.9cm 12.5cm}]{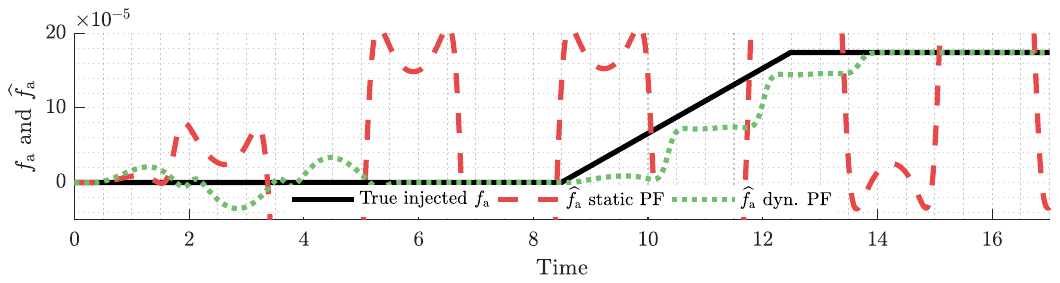}\label{fig:nlprefcomp2}}\\    
     	\subfloat[Multiplicative fault $\fm$]{\includegraphics[scale=0.8, trim={2cm 12.5cm 0.9cm 12.5cm}]{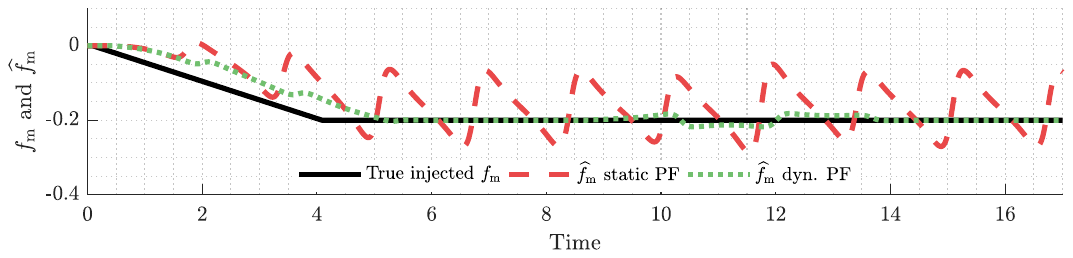}\label{fig:nlprefcomp3}}\\
     	\subfloat[True estimation error in comparison with the corresponding performance bounds in~\eqref{eq:cor:perf I} and~\eqref{eq:cor:perf II} for the static and dynamic pre-filters, respectively.]{\includegraphics[scale=0.8, trim={2cm 12.5cm 0.9cm 12.5cm}]{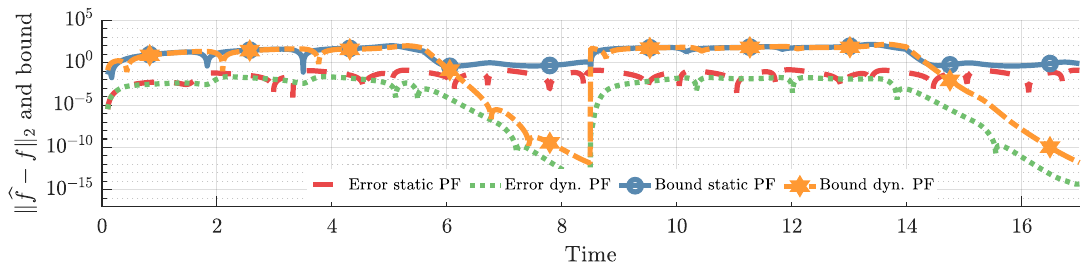}\label{fig:nlprefcomp4}}
     \caption{True estimation errors and the corresponding performance bounds in the presence of constant faults as discussed in Corollaries~\ref{cor:constant faults:I} and~\ref{cor:constant faults:II} (PF: Pre-Filter).}
     \label{fig:reso1}
\end{figure}

Figure~\ref{fig:nlprefcomp2} and~\ref{fig:nlprefcomp3} show the injected faults $\fa$ and $\fm$ as well as the estimated faults $\hfa$ and $\hfm$ using a static and dynamic pre-filter, respectively.  As shown in Figure~\ref{fig:nlprefcomp2} and~\ref{fig:nlprefcomp3}, the  fault estimates $\hfa,\:\hfm$ from the diagnosis filter with static pre-filter (Theorem~\ref{theorem:main:II}) fluctuate heavily  around the true injected fault. This particular behavior is also reflected by the performance bound in the constant-fault case as depicted in Corollary~\ref{cor:constant faults:I}, where we find that in particular the variance of the signal $e$ (in this case study, we consider $E(z)=u$, hence $e$ is either a direct feed-through or filtered version of $u$) is a persistent contributor in the bound of the estimation error. As a matter of fact, the periodic behavior of the estimated faults is found to be oscillating with the same period as input signal $u$.\\ 
The estimated faults that result from the diagnosis filter with the dynamic pre-filter (Theorem~\ref{theorem:main:II}) show convergence to the constant true faults over time. This was to be expected, as it was shown in Corollary~\ref{cor:constant faults:II} that, in the constant fault case, the performance bound decays exponentially as time increases. Figure~\ref{fig:nlprefcomp4} depicts the estimation errors (left-hand side of~\eqref{eq:thm:perf I} and~\eqref{eq:FIfiltpf} for the static pre-filter and the dynamic pre-filter, respectively) and their simulated performance bounds (right-hand side of~\eqref{eq:thm:perf I} and~\eqref{eq:FIfiltpf} for the static pre-filter and the dynamic pre-filter, respectively). It can be observed from this figure that the estimation error and its performance bound for the static pre-filter are time-varying and non-zero, even for constant faults. This confirms our findings from Corollary~\ref{cor:constant faults:I}, where it was shown that the performance bound for constant faults always depends on~$\varnk[]{e}{k-k_0}$ which was assumed to be non-zero. The estimation error shows similar non-zero behavior, showing that the bound gives sufficient insight into the behavior of the estimation error. For the dynamic pre-filter, the performance bound and the estimation error drop towards zero (up to numerical precision) when the fault is constant and the pre-filter has converged. Once a new fault occurs, the estimation error and its performance bound increase and subsequently converge back towards zero. Translating these results back into the application perspective of automated driving shows that, both the additive and multiplicative fault, can be estimated simultaneously in a matter of seconds (depending on the tuning of the horizon $n$ and the poles of the FD filter). This allows the automated vehicle to act upon two different fault types accordingly, as opposed to having to conservatively act on the presence of an aggregated fault signal $\fa+\EZ\fm$ without knowing the separate contributions of the faults. For example, a slight offset $\fa$ in the system could be attenuated in a closed-loop, a large degradation $\fm$ however should result in a transition of the vehicle to a safe state.
\subsection{Sensitivity study}\label{sec:sensitivity}
A sensitivity design is done in this section to touch upon potential degrees of freedom that could affect the estimation error of the fault diagnosis filter. The composed bounds in~\eqref{eq:thm:global I} and~\eqref{eq:thm:global II} consists of elements that usually can not be affected in an open-loop system (i.e., the behavior of the faults $\fa,\:\fm$ and the signal $e$) and those that can be influenced (i.e., the fault detection filter dynamics and the fault isolation time-horizon $n$). In this sensitivity study, the dynamics of the detection filter, as well as the horizon $n$ of the isolation filter, are considered as "tuning" parameters. For the fault detection filter dynamics, only the dominant pole $p$ is considered as a tuning parameter for the sake of simplicity.
\begin{figure}[!t]
     \subfloat[$n=10$, $p=0.6$]{\includegraphics[scale=0.9, clip]{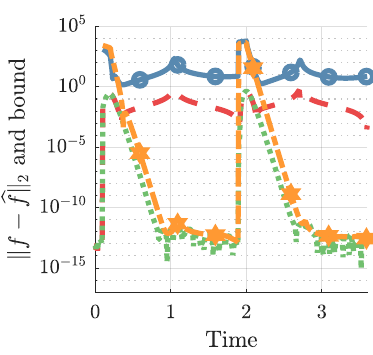}\label{fig:reso_log21}}
     \subfloat[$n=10$, $p=0.85$]{\includegraphics[scale=0.9, clip]{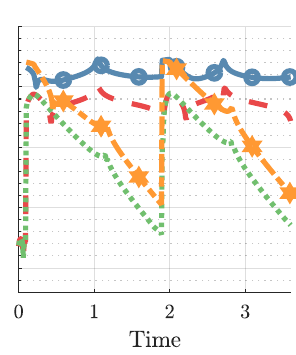}\label{fig:reso_log22}}
     \subfloat[$n=80$, $p=0.85$]{\includegraphics[scale=0.9, clip]{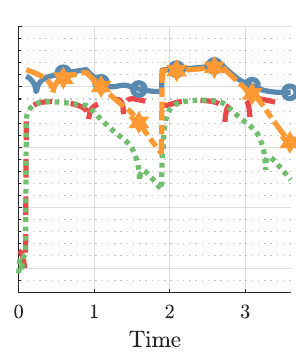}\label{fig:reso_log23}}\\
     \centering
     \includegraphics[trim={0cm 0.4cm 0.9cm 0cm}, scale=0.9]{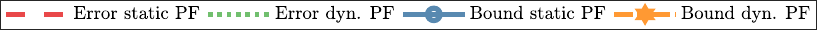}
     \caption{Sensitivity of true estimation errors and the corresponding performance bounds proposed by Corollaries~\ref{cor:constant faults:I} and~\ref{cor:constant faults:II} concerning the dominant pole~$p$ and the regression horizon~$n$ (PF:Pre-Filter).}
     \label{fig:reso_log2}
\end{figure}
Figure~\ref{fig:reso_log2} provides the results of the sensitivity analysis. In Figure~\ref{fig:reso_log22}, the baseline simulation that was shown in Figure~\ref{fig:nlprefcomp4} is provided as means of comparison. In Figure~\ref{fig:reso_log21}, the dominant pole location is shifted from $p=0.85$ to $p=0.6$. From the composed performance bounds~\eqref{eq:thm:global I} and~\eqref{eq:thm:global II} it is clear that, for the performance bound, this strengthens the exponential decay of the average behavior of the faults which can, in particular, be observed in the convergence behavior of the dynamic pre-filter estimation and bound. When increasing the time-horizon $n$ from $n=10$ to $n=80$ it is clearly observed in Figure~\ref{fig:reso_log23} that the bound is pushed down and, as a result, the estimation error is also pushed down. This was to be expected, as the horizon $n$ forms a common denominator for a large part of both performance bounds. However, the exponential decay of the average behavior is slower due to its influence on the average behavior in~\eqref{eq:thm:global I} and~\eqref{eq:thm:global II}. There are more, less trivial, phenomena caused by a change in parameters $n$ and $p$. 
The effects of these phenomena will not be investigated in this analysis as the effects are highly dependent on the signal properties of signals $\fa,\:\fm$ and $e$ and it is therefore difficult to assess generically.
\section{Conclusion and future directions}\label{sec:conclusion}
In this work, a fault diagnosis architecture for the detection, isolation and estimation of additive and multiplicative faults is presented. System-theoretic bounds have been developed for the regression operator. These bounds, including several preliminary lemmas, have proven to be valuable building blocks for the composition of guaranteed performance bounds for the detection, isolation, and estimation architecture for two particular variants of the pre-filter block. The insights gained from the first pre-filter variant (i.e., the identity block) and its behavior for constant faults have been used for the design of a second variant (i.e., the dynamic pre-filter) which has been proven to have an asymptotically decaying bound for constant faults. Simulation results in the domain of SAE level 4 automated driving show the practical value and the potential of the proposed approach in relevant future-proof applications.

A future research direction is the closed-loop implementation of the proposed diagnosis filter. We envision that this could be a stepping-stone toward a class of fault-tolerant control systems that enjoys rigorous stability and performance guarantees. Another interesting direction is the application of the proposed diagnosis architecture in an active fault diagnosis context. This could also build on the key insights obtained from the performance bounds with a particular focus on the role of signals characteristics. Moreover, it is also worth noting that a probabilistic treatment towards the effects of measurements noise (e.g., the $\mathcal{H}_2$ gain of the noise to residual as in~\cite{Svetozarevic20134453}) could also be another direction.

    \bibliographystyle{plain}
    \bibliography{library/main_TN_revision}
\end{document}

%% file: style/style_PM_2.tex
\renewcommand\thesection{\arabic{section}}
\renewcommand\thesubsection{\thesection.\arabic{subsection}}

\makeatletter
\def\@settitle{\begin{center}%
		\baselineskip14\p@\relax
		\normalfont\LARGE\scshape\bfseries
		\@title
	\end{center}%
}
\makeatother

\makeatletter

\def\subsection{\@startsection{subsection}{2}%
	\z@{.5\linespacing\@plus.7\linespacing}{.5\linespacing}%
	{\normalfont\large\bfseries}}

\def\subsubsection{\@startsection{subsubsection}{3}%
	\z@{.5\linespacing\@plus.7\linespacing}{.5\linespacing}%
	{\normalfont\itshape}}

\usepackage[colorlinks	= true,
			raiselinks	= true,
			linkcolor	= MidnightBlue,
			citecolor	= Mahogany,
			urlcolor	= ForestGreen,
			pdfauthor	= {Chris van der Ploeg},
			pdftitle	= {},
			pdfkeywords	= {},   
			pdfsubject	= {},
			plainpages	= false]{hyperref}
\usepackage[usenames, dvipsnames]{color}
\usepackage{dsfont,amssymb,amsmath,subfig, graphicx,fancyhdr,mdframed}
\usepackage{amsfonts,dsfont,mathtools, mathrsfs,amsthm,wrapfig} 
\usepackage[]{siunitx}
\usepackage{ragged2e}
\let\labelindent\relax
\usepackage{enumitem}

\allowdisplaybreaks
\usepackage{algorithm}
\usepackage[noend]{algpseudocode}

\allowdisplaybreaks
\date{\today}

\patchcmd\@setauthors
  {\MakeUppercase{\authors}}
  {\authors}
  {}{}

%% file: style/commands_PM.tex
\newtheorem{Thm}{Theorem}[section]
\newtheorem{Prop}[Thm]{Proposition}
\newtheorem{Fact}[Thm]{Fact}
\newtheorem{Lem}[Thm]{Lemma}
\newtheorem{Cor}[Thm]{Corollary}
\newtheorem{As}[Thm]{Assumption}

\newtheorem*{pprb}{Problem statement}

\newtheorem{Def}[Thm]{Definition}

\theoremstyle{remark}
\newtheorem{Rem}[Thm]{Remark}

\theoremstyle{remark}

\theoremstyle{definition}
\newtheorem*{nnot}{Notation}

\newcommand{\R}{\mathbb{R}}

\newcommand{\N}{\mathbb{N}}

\newcommand{\ra}{\rightarrow}

\newcommand{\ind}{\mathds{1}}
\newcommand{\Let}{\coloneqq}

\newcommand{\tr}{^\intercal}

\newcommand{\C}{\mathbb{C}}
\newcommand{\B}{\mathbb{B}}
\newcommand{\Y}{\mathbb{Y}}

\newcommand{\vark}[2][]{V_n^{#1}\left[#2\right]}

\newcommand{\ave}[2][]{\mu_n^{#1}\left[#2\right](k)}
\newcommand{\avek}[2][]{\mu_n^{#1}\left[#2\right]}
\newcommand{\EZo}[1][]{E(z #1)}
\newcommand{\EZ}[1][]{e}
\newcommand{\fEZ}[1][]{e}
\newcommand{\fa}{f_{{\rm a}}}
\newcommand{\fm}{f_{{\rm m}}}
\newcommand{\hfa}{\widehat{f}_{\rm a}}
\newcommand{\hfm}{\widehat{f}_{\rm m}}

\newcommand{\phact}[1]{\phi_n[#1](k)}

\newcommand{\phactk}[1]{\phi_n[#1]}
\newcommand{\phacttk}[1]{\phi_n^{\intercal}[#1]}
\newcommand{\Phactk}[1]{\Phi_n[#1]}
\newcommand{\Phactpsk}[1]{\phi_n^\dagger[#1]}
\newcommand{\Phactpstk}[1]{\phi_n^{\dagger\intercal}[#1]}

\newcommand{\widebar}[1]{\overline{#1}}
\newcommand{\vecsig}[1]{\mathbf{#1}_n(k)}
\newcommand{\vecsigk}[1]{\mathbf{#1}_n}
\newcommand{\aven}[3][]{\mu_{#3}^{#1}\left[#2\right](k)}
\newcommand{\avenk}[3][]{\mu_{#3}^{#1}\left[#2\right]}
\newcommand{\varn}[3][]{V_{#3}^{#1}\left[#2\right](k)}
\newcommand{\varnk}[3][]{V_{#3}^{#1}\left[#2\right]}
\newcommand{\constant}[1]{\mathcal{C}_{#1}}
\newcommand{\shiftop}{\mathfrak{q}}
\newcommand{\step}[1]{\mathfrak{U}_{#1}}
\newcommand{\tf}[1]{\mathcal{#1}}